\documentclass[a4paper,12pt]{amsart}
\usepackage[utf8]{inputenc}
\usepackage[T1]{fontenc}
\usepackage[UKenglish]{babel}
\usepackage[margin=18mm]{geometry}

\usepackage{color}%cyan,red;magenta,green;yellow,blue

\usepackage{graphicx}

\usepackage{amsmath,amssymb,amsfonts,amsthm}
\usepackage{mathrsfs,eucal,dsfont}
\usepackage{verbatim,enumitem}

\usepackage{hyperref,url}

\renewcommand{\leq}{\leqslant}
\renewcommand{\geq}{\geqslant}

\newcommand{\R}{\mathds R}

\def\aa{\alpha}

\def\d{{\rm d}}
\def\<{\langle}
\def\>{\rangle}

\def\LL{\Lambda}

 \def\ss{\sqrt}
\def\bb{\beta}

\def\R{\mathbb R}   \def\ss{\sqrt} 
  
  \def\vv{\varepsilon} 
\def\<{\langle} \def\>{\rangle}  
  \def\nn{\nabla}  
\def\d{\text{\rm{d}}} \def\bb{\beta} \def\aa{\alpha} 
  \def\si{\sigma} 
 \def\beq{\begin{equation}}  \def\F{\mathscr F}
 
\def\e{\text{\rm{e}}}  \def\OO{\Omega}  
  
 \def\P{\mathbb P} 
\def\C{\mathscr C}           
   
  \def\ll{\lambda}
 
\def\E{\mathbb E} 
  \def\LL{\Lambda}
   
\def\to{\rightarrow}
\def\8{\infty}\def\3{\triangle}
\def\1{\lesssim}

\renewcommand{\hat}{\widehat}

\newtheorem{theorem}{Theorem}[section]
\newtheorem{lemma}[theorem]{Lemma}
\newtheorem{proposition}[theorem]{Proposition}
\newtheorem{corollary}[theorem]{Corollary}

\theoremstyle{definition}
\newtheorem{definition}[theorem]{Definition}

\newtheorem{remark}[theorem]{Remark}

\numberwithin{equation}{section}
\begin{document}
\allowdisplaybreaks

\title[SDEs with non-uniform dissipativity]{Random periodic solutions for stochastic differential equations with non-uniform dissipativity}

\author{
Jianhai Bao\qquad\,
\and\qquad
Yue Wu}
\date{}
\thanks{\emph{J.\ Bao:} Center for Applied Mathematics, Tianjin University, 300072  Tianjin, P.R. China. \url{jianhaibao@tju.edu.cn}}
\thanks{\emph{Y.\ Wu:}
Department of Mathematics and Statistics, University of Strathclyde, 26 Richmond St, Glasgow G1 1XH,  UK. \url{yue.wu@strath.ac.uk}}

\maketitle

\begin{abstract}
This paper is concerned with the existence and uniqueness of random periodic solutions for stochastic differential equations (SDEs), where the drift terms involved need not to be uniformly dissipative.   On the one hand,
via the reflection coupling approach, we investigate the existence of random periodic solutions in the sense of distribution for SDEs without memory, where the drifts are merely {\it dissipative at long distance}. On the other hand, via the synchronous coupling strategy, we establish respectively  the existence of pathwise random periodic solutions for functional SDEs with a finite time lag and an infinite time lag, in which the drifts  are only {\it dissipative on average} rather than uniformly dissipative with respect to the time parameters.

\medskip


\smallskip

\end{abstract}

\section{Introduction and main results}
Since the pioneer work \cite{zhao2009},   random periodic solutions,  describing the widely existing long-term periodic phenomenon, of random dynamical systems have been developed greatly; see e.g.  \cite{DZZ,rpsnumerics2017,luo2014,wu2021,wu2022} for dissipative systems and e.g.  \cite{rpssde2011, rpsspde2011, wu2016, wu2018,zhao2009} for partially dissipative systems.
For the  definitions of pathwise random periodic solutions and random periodic solutions in the sense of distribution, please refer to the Appendix section.

In particular,  based on the foundation built in
\cite{zhao2009}, \cite{rpssde2011} investigated existence of pathwise random periodic solution for a class of semi-linear SDEs with additive noise. Whereafter, \cite{rpsspde2011} extended
\cite{rpssde2011} to semi-linear SPDEs on a bounded
domain with a smooth boundary. Subsequently, \cite{wu2016} and \cite{wu2018} treated anticipating random periodic solutions for SDEs  and SPDEs with multiplicative linear noise, respectively.  For a deterministic system, the traditional approach to establish existence of periodic solutions is to construct firstly a suitable Poincar\'e mapping and then seek out a fixed point. Whereas, such a strategy is no longer powerful for stochastic systems due to the presence of randomness. So far, there are two well-developed ways to handle existence of random periodic solutions for stochastic systems. The pull-back  is one of the potential ways to investigate existence of random periodic solutions for stochastic dissipative systems (e.g. \cite{rpsnumerics2017,luo2014,wu2021,wu2022}). Moreover, the stable adapted random periodic solutions can be shown as limits of the pull-back semi-flows generated by the SDEs/SPDEs involved; see, for instance,   \cite{rpsnumerics2017}. When the stochastic systems under consideration is merely partially dissipative, the pull-back approach does not work. In such case, the method based on a generalized Schauder's fixed point theorem and the Wiener-Sobolev compact embedding was put forward and applied in e.g. \cite{rpssde2011, rpsspde2011, wu2016, wu2018,zhao2009} to investigate existence of random periodic solutions for semi-linear SDEs/SPDEs. Furthermore, the unstable anticipating random periodic solution can be identified as a solution of a coupled forward-backward infinite horizon stochastic integral equation. Note that the latter method does not guarantee the uniqueness of random periodic solutions.

The partial dissipativity in e.g. \cite{rpssde2011, rpsspde2011, wu2016, wu2018,zhao2009} requires that the linear term in the (periodic) drift is dissipative in some directions while it is non-dissipative in the other directions.  Let $b_1:\R\times\R^d\to \R^d$ be dissipative. Now we perturb $b_1$ by a bounded measurable function $b_0:\R\times\R^d\to\R^d$. It is easy to see that the drift term $b:=b_0+b_1$ need not to be dissipative in any directions at   short distance. To be precise, we are interested in such a partial dissipativity, which allows the drift term to be dissipative only at long range whereas non-dissipative at short distance; see the condition \eqref{HH} below for more details. Let us call this partial dissipativity the \emph{dissipativity at long distance}. In the first instance, we sought to investigate the weak existence and uniqueness of random periodic solutions for time-periodic SDEs with dissipativity at long distance.

To motivate the idea of dissipativity on average, let us revisit the uniform dissipative condition adopted in \cite{rpsnumerics2017, luo2014,wu2021,wu2022}. More precisely, for the periodic drift term $b_1:\R\times \R^d\to \R^d$, there exists a constant $\ll>0$ such that
\begin{equation}\label{E0}
\<x-y,b_1(t,x)-b_1(t,y)\>\le-\ll|x-y|^2,\qquad x,y\in\R^d,\quad  t\in\R.
\end{equation}
Note that the periodicity is hindered in the dissipative condition above. To reinforce the periodicity, it is natural to replace
the condition \eqref{E0} by the following one
\begin{equation} \label{E00}
\<x-y,b_1(t,x-b_1(t,y))\>\le\ll(t)|x-y|^2,\qquad x,y\in\R^d,\quad t\in\R,
\end{equation}
for some periodic function  $\ll:\R\to\R.$ Moreover, the mean of $\lambda(\cdot)$ during one period is assumed to be negative (i.e., dissipative), which is  also termed as \emph{dissipativity on average}. This allows $\ll(\cdot)$
to be positive at some time points so that the drift term is non-dissipative in any directions. Compared with \eqref{E0}, the condition \eqref{E00} reflects fairly the periodic property of the drift $b_1.$  In the second phase, provided that the drift term is dissipative on average, we attempt to explore
whether the time-periodic SDE under consideration, namely  functional SDEs with a finite lag and functional SDEs with an infinite time lag,   admits a unique pathwise random periodic solution.

No matter dissipativity at long distance or dissipativity on average, the drift terms under consideration are allowed to be non-dissipative in any direction at short distance or at some time points, as shown in the following   subsections.

\subsection{Random periodic solution for  SDEs: dissipativity at long distance} \label{sec:dald}
For a subinterval  $U\subset \R$, let $C(U;\R^d)$ be the collection of continuous $\R^d$-valued functions on $U$.
 Let $\OO=C_0(\R;\R^d)$,
where
$$C_0(\R;\R^d):=\big\{\omega\in C(\R;\R^d):\omega(0)={\bf 0}\big\}.$$
For each $t\in\R$, let $\pi(t):C(\R;\R^d)\to\R^d$ be the projection operator defined by $\pi(t)\xi=\xi(t), \xi\in C(\R;\R^d). $
Now, we equip $C_0(\R;\R^d)$ with the $\si$-algebra
 $$\F:=\si(\pi(s)|s\in\R )$$
and the filtration
$$\F_t:=\si(\pi(s)|s\le t),\qquad t\in\R.$$
Let  $\P$ be the two-sided $d$-dimensional Wiener measure on $(\OO,\F)$, which obviously is  a measure-preserving probability. Let $\theta$ be the Wiener shift operator defined by $(\theta_t\omega)(s)=\omega(t+s)-\omega(t)$ for all $s,t\in\R$ and $\omega\in\OO.$ Then, $(\OO,\F,\P,\theta)$ is a metric dynamical system.
 For each $\omega\in\OO$ and $t\in\R$, define $W(t,\omega)=\omega(t)$. Then, $(W(t))_{t\in\R}$ is a  $d$-dimensional Brownian motion on the probability space $(\OO,\F,\P).$ For a vector valued  (or matrix-valued)  function $f$ on $\R$ and a positive constant $\tau$, $f$ is said to be $\tau$-periodic if $f(t+\tau)=f(t)$
 for all $t\in\R.$  For  random variables $\xi$ and $\eta$, we write $\xi\overset{d}{=}\eta$ to demonstrate  that they have the same law. Set $\Delta:=\{(t,s)\in \mathbb{R}^2, t\ge s  \}$.

To showcase the dissipativity at long distance, we take the SDE  with an additive noise as a toy example:
\begin{equation}\label{H*}
 \d X(t)=b(t,X(t))\,\d t+\ss{\aa(t)}\, \d W(t),\qquad t\ge s\in\R,
\end{equation}
in which
$$b:\R\times\R^d\to\R^d,\qquad   \aa:\R \to(0,\8).$$

Assume that
\begin{enumerate}
\item[$({\bf A})$] For each fixed $x\in\R^d$, $b(\cdot,x)$ and $\aa(\cdot)$ are $\tau$-periodic and continuous on $\R$. Moreover, $b$ is bounded on bounded sets of $\R\times \R^d$ and there exist   constants $K_1,L\ge0,K_2>0$ such that for all $t\in\R$ and $x,y\in\R^d,$
\begin{equation}\label{HH}
 \<x-y,b(t,x)-b(t,y)\>\le\aa(t)\big(K_1|x-y|^2{\bf 1}_{\{|x-y|\le L\}}-K_2|x-y|^2  {\bf1}_{\{|x-y|> L\}}\big).
\end{equation}
\end{enumerate}

Under $({\bf A})$, in terms of \cite[Theorem 3.1.1]{PR}, \eqref{H*} has a unique strong solution  $(X^{s,x}(t))_{t\ge s}$ with the initial  value  $x\in\R^d$ at the starting time $s\in\R.$
Then, the mapping $\phi:\triangle \times\R^d\times \OO\to\R^d$ defined by
\begin{equation}\label{0j}
(t,s,x,\omega)\mapsto \phi(t,s,x,\omega)=X^{s,x}(t,\omega)
\end{equation}
is a stochastic semi-flow.

Our first main result in this paper is presented as follows.

\begin{theorem}\label{main2}
Under Assumption $({\bf A})$, the stochastic semi-flow $\phi$, defined by \eqref{0j}, has a unique  random $\tau$-periodic solution in the sense of distribution, i.e.,
there exists a stochastic process  $(X^*(t))_{t\in\R}\in L^1(\OO\to\R^d,\F,\P)$  such that for all $(t,s)\in\triangle, h\ge0,$ and $\xi=\xi_s(\cdot)=\xi_0(\theta_{-s}\cdot)\in L^1(\OO\to\R^d,\F_s,\P)$,
\begin{equation*}
X^*(t+h,\omega)\overset{d}{=}\phi(t+h, t, X^*(t,\omega),\omega),\qquad X^*(t+\tau,\omega)\overset{d}{=}X^*(t,\theta_\tau\omega),
\end{equation*}
and
$$\lim_{s\to-\8}X^{s,\xi}(t)\overset{d}{=}X^*(t).$$
\end{theorem}

Below, we make some comments on Theorem \ref{main2}, in particular, concerned with the Assumption ({\bf A}) and the framework \eqref{H*}.

\begin{remark}
It is trivial to see that the condition \eqref{HH} in Assumption $({\bf A})$ goes back to the  uniformly dissipative condition in case of $K_1=L=0$. In particular,  the uniformly dissipative condition is imposed in \cite{wu2021}  to establish existence and uniqueness of random periodic solutions for a class of semi-linear SDEs.
On the other hand, the condition \eqref{HH} shows that the drift term $b$ is merely dissipative at   long distance whereas it need not to be dissipative at the short distance; see, for instance, $b(t)=\aa(t)(x-x^3), x\in\R,$ for some $\aa:\R\to (0,\8)$.  Hence, the condition \eqref{HH} is much weaker than the uniformly dissipative condition. If the drift term of the SDE  involved is uniformly dissipative, we can exploit  the   synchronous coupling approach \cite{Chen}
to obtain existence and uniqueness of random periodic solutions; see, e.g., \cite{rpsnumerics2017,wu2021}. However, provided that the drift term of the SDE we work  on is non-uniformly dissipative (e.g. Assumption ({\bf A})), the statement based on synchronous coupling strategy is therefore violated. As for such setting, we invoke the reflection coupling method to handle the difficulty arising from the non-uniform dissipativity of drifts as showed in the proof of Theorem \ref{main2}. In contrast to  \cite{rpsnumerics2017,wu2021}, we herein establish the weak existence and uniqueness (i.e. in the sense of distribution) rather than the pathwise existence and uniqueness of random periodic solutions for SDEs with  non-uniformly dissipative drifts.
\end{remark}

\begin{remark}
For the sake of succinctness, in Theorem \ref{main2}
we treat  only   the case of additive noises to
elaborate the role of reflection coupling in establishing existence of random periodic solutions for SDEs with non-uniformly dissipative drifts. Whereas, by a close inspection of the argument for  Theorem \ref{main2}, it can be extended without essential difficulties to certain setup of multiplicative noises. In particular, for the setting that  the diffusion term $\si$ can be decomposed via the relationship:  $$(\si\si^*)(t,x)=\aa(t)I_{d\times d}+(\si_0\si_0^*)(t,x),\qquad t\in\R,~x\in\R^d$$ for some $\aa :\R\to(0,\8)$ and $\si_0:\R\times \R^d\to\R^d\otimes\R^d$,  by  applying  the
reflection coupling to the additive noise and the synchronous coupling to the multiplicative noise (see, for example, \cite{RSW,Wang}),  then the assertion in Theorem \ref{main2} remains  valid.
\end{remark}

\subsection{Random periodic solutions for functional SDEs: dissipativity on average}
As aforementioned, dissipativity on average is a slight generalisation of the uniformly dissipative condition. It allows the periodic drift associated with the functional SDE to be non-dissipative at some time within one period. Under such condition, we consider the existence and uniqueness of pathwise random periodic solutions for the functional SDE as a general case. Besides, we drop the Lipschitz condition imposed on the drift term in \cite{luo2014}. The functional SDE describes the evolution of a random system whose future status not only depend on its presence but also on its past, thus it generalises the classical SDE without memory. Such equations are used to model random processes with a memory, where the memory can be a finite time lag or an infinite time lag. Because of the variability on the memory, the analysis for the existence and uniqueness of the pathwise random periodic solution under the two cases will be a little bit different (as shown in Lemma \ref{ll3}) and thus discussed separately.
\subsubsection{Functional SDEs with a finite time lag: dissipativity on average}

We further introduce additional notation.
For a fixed constant $r_0>0$ representing the length of the time lag, set $\mathscr C:=C([-r_0,0];\R^d)$, which is a Polish space endowed with the uniform norm $\|f\|_\8:=\sup_{-r_0\le u\le0}|f(u)|$. With regard to  the function
$f\in C([0,a],\R)$ for some constant $a\ge0,$
we write
$\|f\|_{a,\8}:=\sup_{0\le s\le  a}|f(s)|$.
For each  fixed $t\in\R$
and  $f\in C(\R;\R^d)$,
we define  $f_t\in\C$ by $f_t(u)=f(t+u),u\in[-r_0,0]$, i.e., $f_t$ is a path segment on the interval $[t-r_0,t]$.

In this subsection, we are
  interested in the following functional SDE
\begin{align}
  \label{eq:SPDE}
    \d X(t) =
   b(t,X_t)\,\d t+\si(t,X_t)\,\d W(t), \qquad t\ge s\in\R,
\end{align}
  where
  $$b:\R\times \C\to\R^d, \qquad  \si: \R\times \C\to \R^d\otimes\R^d$$
are measurable.

Concerning the drift term  $b$ and the diffusion term $\si$, we assume that
\begin{enumerate}
\item[({\bf H})] For each $\xi\in\C$, $b(\cdot,\xi)$ and $\si(\cdot,\xi)$ are $\tau$-periodic and continuous on $\R$. Moreover, $b$ is uniformly  bounded on each bounded set of $\R\times\C$ and there exist continuous $\tau$-periodic functions $\lambda_1:\R\to\R, \lambda_2,\ll_3:\R\to[0,\8)$ such that for any $\xi,\eta \in \mathscr{C}$ and $t\in \R$,
  \begin{align}
  &2\langle b(t,\xi)-b(t,\eta),\xi(0)-\eta(0) \rangle \leq \lambda_1(t) |\xi(0)-\eta(0)|^2+\lambda_2(t) \|\xi-\eta\|^2_\infty\label{eqn:A1},\\
   & \|\si(t,\xi) - \si(t,\eta)\|^2_{\rm HS} \le \lambda_3(t) \|\xi - \eta\|^2_\8,\label{eqn:A2}
  \end{align}
  where, for a matrix $A$, $\|A\|_{\rm HS}$ stands for its Hilbert-schmidt norm.
\end{enumerate}

Under Assumption ({\bf H}), according to \cite[Theorem 2.3]{RS}, the functional SDE \eqref{eq:SPDE} has a unique functional solution  $(X_t)_{t\ge s}$. In terminology, the functional solution $(X_t)_{t\ge s}$ is also called the solution path segment or the window process associated with the solution process $(X(t))_{t\ge s}$. There are some different essential  features between $(X(t))_{t\ge s}$ and $(X_t)_{t\ge s}$. For instance, the solution process $(X(t))_{t\ge s}$ is finite dimensional while the segment process $(X_t)_{t\ge s}$ is infinite dimensional since the corresponding state space belongs to a function space; the solution process $(X(t))_{t\ge s}$ is indeed a semi-martingale whereas  the functional solution $(X_t)_{t\ge s}$ is not; the solution process $(X(t))_{t\ge s}$ does not enjoy the semi-flow property nevertheless the functional solution $(X_t)_{t\ge s}$ admits the semi-flow property as showed below;  the solution process $(X(t))_{t\ge s}$ is not Markovian however the segment path $(X_t)_{t\ge s}$ is,
to name a few.
For the classical  monographs  on functional SDEs, we refer to Mohammed \cite{Mo} upon Markov property, trajectory properties and infinitesimal generator in $L^2$-space, and so forth, and Mao \cite{Mao} for wellposedness and  stability analysis.

In this context, we shall work on the infinite dimensional functional solution $(X_t)_{t\ge s}$ rather than the finite dimensional solution process $(X(t))_{t\ge s}$.
Below, we shall write $X_t^{s,\xi}$ in lieu of $X_t$ to emphasize the functional solution $X_t$ starting from the initial value $\xi\in\C$ at the starting time $s\in\R$. Define the    solution path segment  $\phi:$  $\triangle\times \C\times\OO\to\C$ by
\begin{equation}\label{WWW}
\phi(t,s,\xi,\omega)=X_t^{s,\xi}(\omega)=(X^{s,\xi}(u,\omega))_{t-r_0\le u\le t}.
\end{equation}
Note that $\phi$ is jointly measurable and, for all $s\in\R$ and $\omega\in\Omega$, $\phi(s,s,\cdot,\omega)={\rm id}_\C$, where ${\rm id}_\C$ stands for the identity operator on $\C.$
Moreover, due to the strong wellposedness (see e.g. \cite[Theorem 2.3]{RS}) of \eqref{eq:SPDE}, we obviously have
\begin{equation*}
X_t^{s,\xi}=X_t^{r,X_r^{s,\xi}},\qquad s\le r\le t,\quad \xi\in\C.
\end{equation*}
Whence, we obtain that for all  $s\le r\le t$ and $\omega\in\Omega$,
\begin{equation*}
\phi(t,s,\cdot,\omega)=\phi(t,r,\phi(r,s,\cdot,\omega),\omega).
\end{equation*}
 Therefore, we conclude that  the path map defined in \eqref{WWW} is a stochastic semi-flow.

\bigskip
Our second main result in this work is stated as below.

\begin{theorem}\label{thm:main1}
Assume Assumption $({\bf H})$ and suppose further
\begin{equation} \label{WE}
   \ell:=\int_0^\tau\big(\ll_1(r)+2\e^{-c_*(r_0,\tau)}\big(\ll_2(r)+\ll_3(r)+2\ll_3(r) \chi^2\e^{-c_*(r_0,\tau)+2c^*(r_0,\tau)}\big)\big)\,\d r<0,
\end{equation}
where $\chi\approx1.30693$ is the optimal Burkholder-Davis-Gundy upper constant in Lemma \ref{ose} below, and
\begin{align}        \label{WW}
    \begin{split}
c_*(r_0,\tau):=\inf_{0\le u\le \tau,-r_0\le \theta\le0}\int_{u+\theta}^u\ll_1(s)\,\d s, \qquad  c^*(r_0,\tau):=\sup_{0\le u\le \tau,-r_0\le \theta\le0}\int_{u+\theta}^u\ll_1(s)\,\d s.
    \end{split}
\end{align}
Then, the stochastic semi-flow $\phi$, defined in \eqref{WWW}, has a
 a unique pathwise random $\tau$-periodic functional solution. That is, there exists a unique stochastic process $ (X^{*}_t )_{t\in\R}\in L^2(\Omega\to\C,\F,\P)$ such that  for all $t\in\R$ and $h\ge0$,
 \begin{equation*}
X^*_{t+h}(\omega)=\phi(t+h, t, X^*_t(\omega),\omega),\qquad X^*_{t+\tau}(\omega)=X^*_t(\theta_\tau\omega)\quad {\rm a.s. }
\end{equation*}
 and for all $(t,s)\in\triangle$ and $ \xi=\xi_s(\cdot)=\xi_0(\theta_{-s}\cdot) \in L^2(\Omega\to\C,\F_s,\P)$,
\begin{equation}\label{eqn:lim_sde}
    \lim_{s\to-\8}\E\|X^{s,\xi}_t-X^*_t\|_\infty^2=0.
\end{equation}
\end{theorem}
\smallskip

For any $t\in\R$, if $\ll_1(t)\equiv-\ll_1$, and $\ll_2(t)\equiv\ll_2$ for some constants $\ll_1>0,\ll_2\ge0$, then the assumption  \eqref{eqn:A1} reduces  to the classical dissipative-type assumption.
Furthermore, we assume  $\ll_3(t)\equiv\ll_3$ for some $\ll_3\ge0$. In this setting, by noting $c_*(r_0,\tau)= -\ll_1 r_0$ and $c*(r_0,\tau)=0$.  Thus, we obtain the following corollary.
\bigskip
\begin{corollary}
Assume Assumption $({\bf H})$ with $\ll_1(t)\equiv-\ll_1, \ll_2(t)=\ll_2$ and $\ll_3(t)=\ll_3$ for some $\ll_1>0,\ll_2,\ll_3\ge0$ and suppose further
\begin{equation}\label{EW}
 \ll_1>2\e^{ \ll_1 r_0}\big(\ll_2+\ll_3+2\ll_3\chi^2\e^{ \ll_1 r_0}\big),
\end{equation}
where $\chi\approx1.30693$.
Then, the stochastic semi-flow $\phi$, defined in \eqref{WWW}, admits a
 unique pathwise random $\tau$-periodic solution $ (X^{*}_t)_{t\in\R} \in L^2(\Omega\to\C,\F,\P)$ satisfying   for all $(t,s)\in\triangle$ and $ \xi =\xi_s(\cdot)=\xi_0(\theta_{-s}\omega) \in L^2(\Omega\to\C,\F_s,\P)$,
\begin{equation*}
    \lim_{s\to-\8}\E\|X^{s,\xi}_t-X^*_t\|_\infty^2=0.
\end{equation*}
\end{corollary}

\bigskip
 \begin{remark}
 When the drift terms  are uniformly dissipative, existence of random periodic functional solutions was addressed in \cite{luo2014} for functional SDEs with {\it Lipschitz continuous coefficients}.
 In Theorem \ref{thm:main1},     the functional SDE involved need not to be uniformly  dissipative   pointwise w.r.t. the time  variable.  In other words, the functional SDE we are interested in is allowed to be {\it non-dissipative at some time points}.
More precisely, in terms of \eqref{WE},  we require merely the functional SDE we are handling  are {\it dissipative  on average}. For classical SDEs without memory, the uniformly dissipative condition   in e.g. \cite{rpsnumerics2017,wu2021} is uniform with respect to the time variables. In certain sense, the uniformly dissipative condition imposed in \cite{rpsnumerics2017,wu2021} does not embody very well the periodic property of coefficients. Comparing with the existing literature e.g. \cite{rpsnumerics2017,wu2021},
 our Theorem \ref{thm:main1} is new even for SDEs without memory.
 \end{remark}

\smallskip
\begin{remark}
For classical SDEs without memory, which are dissipative at long distance, we apply the reflection coupling approach to establish existence of random periodic solutions. However, concerning functional SDEs, the reflection coupling method no longer work since the state space of the functional solution is a function space. So, instead of the reflection coupling strategy, we invoke the synchronous coupling (see Lemma \ref{lem:stable1} for more details) to investigate existence of random periodic functional solutions to functional SDEs, which are dissipative {\it on average} rather than uniformly dissipative.
\end{remark}

\subsubsection{Functional SDEs with an infinite time lag: dissipativity on average}
In this subsection, we aim to extend Theorem \ref{thm:main1} to functional SDEs with an infinite time lag; see e.g. \cite{MYW,WYM} for more backgrounds and long-term behavior.
For this purpose, we introduce the following set:
for a fixed constant $\aa_0>0$,
\begin{equation*}
 \C_{\aa_0} :=\Big\{\xi\in C((-\8,0];\R^d):\|\xi\|_{\aa_0}:= \sup_{-\8<\theta\le0}(\e^{\aa_0\theta}|\xi(\theta)|)<\8\Big\},
\end{equation*}
which is a Polish space by equipping the metric induced by $\|\cdot\|_{\aa_0}$.

In this subsection, we still work on \eqref{eq:SPDE} but with the finite time lag replaced by the infinite one. Namely, we consider
the following functional SDE
\begin{align}
  \label{SDE}
    \d X(t) =
   b(t,X_t)\,\d t+\si(t,X_t)\,\d W(t), \qquad t\ge s\in\R,
\end{align}
  where
  $$b:\R\times \C_{\aa_0}\to\R^d, \qquad  \si: \R\times \C_{\aa_0}\to \R^d\otimes\R^d$$
are measurable.

Below, we suppose that
\begin{enumerate}
\item[$({\bf H}')$] For each $\xi\in\C_{\aa_0}$, $b(\cdot,\xi)$ and $\si(\cdot,\xi)$ are $\tau$-periodic. Moreover, $b$ is uniformly  bounded on each bounded set of $\R\times\C_{\aa_0}$ and there exist $\tau$-periodic functions $\lambda_1:\R\to\R, \lambda_2,\lambda_3:\R\to[0,\8)$ such that for any $\xi,\eta \in \mathscr{C}_{\aa_0}$ and $t\in \R$,
  \begin{align}
  &2\langle b(t,\xi)-b(t,\eta),\xi(0)-\eta(0) \rangle \leq \lambda_1(t) |\xi(0)-\eta(0)|^2+\lambda_2(t) \|\xi-\eta\|^2_{\aa_0}\label{AA1},\\
   & \|\si(t,\xi) - \si(t,\eta)\|^2_{\rm HS} \le \lambda_3(t) \|\xi - \eta\|^2_{\aa_0}.\label{AA2}
  \end{align}
\end{enumerate}

Under the Assumption $({\bf H}')$, \eqref{SDE} has a unique functional solution $(X_t^{s,\xi})_{t\ge s}$ with the initial value $\xi\in\C_{\aa_0}$ at the time $s\in\R;$ see, for instance, \cite[Theorem A.1]{BWY} for more details. Thus, the solution path segment $\phi:\triangle \times\C_{\aa_0}\times\OO\to \C_{\aa_0}$, defined by
\begin{equation}\label{J*}
\phi(t,s,\xi,\omega)=X_t^{s,\xi}(\omega)=(X^{s,\xi}(u,\omega))_{-\8<u\le t}
\end{equation}
is a stochastic semi-flow.

Our third main result in this paper is described as below.
\begin{theorem}\label{thm3}
Assume $({\bf H}')$ and suppose further that
\begin{equation}\label{BB1}
    \int_0^\tau(\ll_1(u)+2\aa_0)\,\d u\ge0,
\end{equation}
and
\begin{equation}\label{B*}
\int_0^\tau\big(\ll_1(u)+\ll_2(u)+(1+2\chi^2)\ll_3(u)\big)\,\d u<0,
\end{equation}
where $\chi\approx1.30693$ and
\begin{equation}\label{BB}
  \ll_{\aa_0,\tau} :=\sup_{\theta\le0,0\le\aa,\bb\le\tau}\bigg(\frac{1}{\tau} (\theta+\aa)\int_0^\tau(\ll_1(u)+2\aa_0)\,\d u-\int^{\aa+\bb}_{\bb}(\ll_1(u)+2\aa_0)\,\d u\bigg).
\end{equation}
Then, the stochastic semi-flow $\phi$, given  by \eqref{J*}, has a
 a unique pathwise random $\tau$-periodic functional solution  $ (X^{*}_t)_{t\in\R} \in L^2(\Omega\to\C_{\aa_0},\F,\P)$ satisfying   for all $(t,s)\in\triangle$ and $ \xi=\xi_s(\cdot)=\xi_0(\theta_{-s}\omega) \in L^2(\Omega\to\C_{\aa_0},\F_s,\P)$,
\begin{equation*}
    \lim_{s\to-\8}\E\|X^{s,\xi}_t-X^*_t\|_{\aa_0}^2=0.
\end{equation*}
\end{theorem}

As a byproduct of Theorem \ref{thm3}, it is immediate to
 obtain the following corollary.

\begin{corollary}\label{co}
Assume $({\bf H'})$ with $\ll_1(u)\equiv-\ll_1$,
$\ll_2(u)\equiv\ll_2$ and $\ll_3(u)\equiv\ll_3$ for some $\ll_1,\ll_2,\ll_3>0$  and suppose further that
\begin{equation}\label{J**}
    2\aa_0\ge \ll_1>\ll_2 +(1+2\chi^2)\ll_3.
\end{equation}
 Then, the stochastic semi-flow, generated by the functional solution $(X_t^{s,\xi})_{t\ge s}$ to \eqref{SDE}, possesses a unique pathwise random $\tau$-periodic functional solution.
\end{corollary}
\begin{proof}
In the setting of Corollary \ref{co}, it is easy to see that  $\ll_{\aa_0,\tau}=0$ so the proof of Corollary \ref{co} is finished by applying Theorem \ref{thm3} and taking \eqref{J**} into consideration.
\end{proof}

The content of this paper is organized as follows. In Section \ref{sec3}, we provide a general criteria on existence and uniqueness in the sense distribution of random periodic solutions for stochastic semi-flows. As an application, we complete the proof of Theorem \ref{main2}. Section \ref{sec:existence} is devoted to the proofs of Theorems \ref{thm:main1} and \ref{thm3} via a refined principle upon existence and uniqueness of pathwise random periodic solutions for stochastic semi-flows.

\section{Proof of Theorem \ref{main2}}\label{sec3}
To begin, we introduce some additional notation. Let $\mathscr P(\R^d)$ be the collection of probability measures on $\R^d$. For a distance-like function $\psi$ on $\R^d\times\R^d$ (i.e., $\psi:\R^d\times\R^d\to[0,\8)$ is symmetric, lower semi-continuous and such that $\psi(x,y)=0\Leftrightarrow x=y$; see e.g. \cite[Definition 4.3]{HMS}), set
$$\mathscr P_\psi(\R^{d}):=\Big\{\mu\in\mathscr P(\R^d):\int_{\R^d}\psi(x,{\bf0})\,\mu(\d x)<\8\Big\}.$$
Define the quasi-Wasserstein distance
$\mathbb W_\psi$ (see e.g. \cite[(4.3)]{HMS}) by
 \begin{equation*}
 \mathbb W_\psi(\mu,\nu)= \inf_{\pi\in \mathcal C} \int_{\R^d\times \R^d}  \psi(x,y)\pi(\d x,\d y),\qquad \mu,\nu\in\mathscr P_\psi(\R^{d}),
\end{equation*}
where $\mathcal C$ denotes the set of couplings for $\mu$ and $\nu.$ Under $\mathbb W_\psi$, the space $\mathscr P_\psi(\R^{d})$ is complete, that is, every $\mathbb W_\psi$-Cauchy sequence in $\mathscr P_\psi(\R^{d})$ converges under $\mathbb W_\psi$.
In  case of   $\psi(x,y)=|x-y|$,  $\mathbb W_\psi$ is in fact the classical $L^1$-Wasserstein distance. For  this case, in the following part we shall write $\mathbb W_1$ instead of $\mathbb W_\psi.$ On the other hand, if $\psi(x,y)={\bf 1}_{\{x\neq y\}}$, $\mathbb W$ is indeed the total variation distance. In general, for a distance-like function $\psi$ on $\R^d\times\R^d$, the classical triangle inequality associated with $\psi$ might be violated %\yuebox{for example?}
so $\mathbb W_\psi$ need not to be a genuine metric on $\mathscr P_\psi(\R^{d})$.   If the distance-like function $\psi$ on $\R^d\times\R^d$ further satisfies a weak form of the triangle inequality, i.e., there exists a constant $c_0>0$  such that
\begin{equation}\label{WU}
  \psi(x,y)  \le c_0\big(\psi(x,z)+\psi(z,y)\big), \qquad x,y,z\in\R^d,
\end{equation}
then the corresponding quasi-Wasserstein distance $\mathbb W_\psi$  obeys a weak form of the triangle inequality as well.
  Concerning a random variable $\xi$, denote $\mathscr L_\xi$ by its law. In some occasion, for a random variable $\xi$,  we write $\xi\sim\mu\in\mathscr P(\R^d)$ to explicate that the law of $\xi$ is $\mu.$
 Let
$L_\psi(\OO\to\R^d,\F_s,\P)$
be the space of all $\F_s$-measurable $\R^d$-valued random variables $\xi$ such that $\mathscr L_\xi\in\mathscr P_\psi(\R^{d})$.

Before we start to finish the proof of Theorem \ref{main2}, we prepare the following general criteria, which indeed is a weak version of \cite[Theorem 3.2.4]{luo2014}, to establish existence and uniqueness in the sense of distribution of random periodic solutions for stochastic dynamical systems.

\bigskip
\begin{proposition}\label{pr} For  an $\R^d$-valued Markov process $(Y^{s,x}(t))_{t\ge s}$ with the initial value $x\in \R^d$ at the starting point $s\in\R,$
assume that
\begin{enumerate}
\item[$(i)$]$(Y^{s,\xi}(t))_{t\ge s}$ has continuous trajectories and enjoys the semi-flow property;

\item[$(ii)$]For all $\xi=\xi_s\in L_{\psi}(\OO\to\R^d, \F_s,\P)$, there exists a constant $ C_0(\xi)>0$ such that
\begin{equation*}
 \sup_{t\ge s} \mathbb W_\psi\big(\mathscr L_{Y^{s,\xi}(t)},\delta_{\bf0}\big)\le C_0(\xi),
\end{equation*}
where the distance-like function $\psi$ on $\R^d\times\R^d$ satisfies the weak triangle inequality \eqref{WU} and $\delta_{\bf0}$ is the Dirac measure centered at the point ${\bf0}$;

\item[$(iii)$] There exists a   function $h:[0,\8)\to[0,\8)$ satisfying for each fixed $t\in\R$ and some $\tau_0>0$,
\begin{equation}\label{L1*}
    \lim_{s\to-\8}\sum_{j=0}^\8h(t-s+j\tau_0)=0
\end{equation}
and such that for all $(t,s)\in\triangle$ and $ \xi=\xi_s,\eta=\eta_s\in  L_{\psi}(\OO\to\R^d, \F_s,\P)$,
\begin{equation}\label{TT4}
\mathbb W_\psi\big(\mathscr L_{Y^{s,\xi}(t)},\mathscr L_{Y^{s,\eta}(t)}\big)\le h(t-s)\mathbb W_\psi\big(\mathscr L_{\xi},\mathscr L_{\eta}\big).
\end{equation}
\end{enumerate}
Then, if the map $\phi:\triangle \times \R^d\times \OO\to \R^d$ defined via
\begin{equation*}
\phi(t,s,x,\omega)=Y^{s,x}(t,\omega), \quad (t,s)\in\triangle, ~x\in \R^d, ~\omega\in\Omega,
\end{equation*}
further satisfies the following property: for all $(t,s)\in\triangle, x\in \R^d,$   $\omega\in\Omega,$ and some $\tau>0,$
\begin{equation}\label{YYY}
\phi(t+\tau,s+\tau,x,\omega)=\phi(t,s,x,\theta_\tau\omega),
\end{equation}
there exists a unique $($in the sense of law$)$ $\F$-measurable stochastic process $(Y^*(t))_{t\in\R}$ such that for all $t\in\R$ and $h\ge0,$
\begin{equation}\label{T}
Y^*(t+h,\omega)\overset{d}{=}\phi(t+h, t, Y^*(t,\omega),\omega),\qquad Y^*(t+\tau,\omega)\overset{d}{=}Y^*(t,\theta_\tau\omega)
\end{equation}
and   for all $(t,s)\in\triangle$ and $\xi=\xi_s=\xi_0(\theta_{-s}\omega)\in L_\psi(\OO\to\R^d,\F_s,\P)$,
\begin{equation} \label{T***}
\lim_{s\to-\8}Y^{s,\xi}(t)\overset{d}{=}Y^*(t).
\end{equation}
\end{proposition}

\begin{proof}
By virtue of (ii), the quantity $\mathbb W_\psi\big(\mathscr L_{Y^{s,\xi}(t)},\mathscr L_{Y^{s,\eta}(t)}\big)$  for all $(t,s)\in\triangle$ and $\xi=\xi_s,\eta=\eta_s\in  L_{\psi}(\OO\to\R^d, \F_s,\P)$   is well defined.
If $\mathscr L_{Y^{\cdot,\xi}(t)}$ is a $\mathbb W_\psi$-Cauchy sequence in  $\mathscr P_\psi(\R^{d})$, i.e, for all $\xi=\xi_s,\eta=\eta_s\in  L_{\psi}(\OO\to\R^d, \F_s,\P)$ and $t\in\R$,
\begin{equation}\label{Y}
\lim_{s\to-\8}\sup_{r\ge0}\mathbb W_\psi\big(\mathscr L_{Y^{s-r,\xi}(t)},\mathscr L_{Y^{s,\xi}(t)}\big)=0,
\end{equation}
then  $Y^{s,\xi}(t)$
converges weakly as $s\to -\8$ (see e.g. \cite[Theorem 6.9, p.96]{Vi})
to some $\F$-measurable limiting process, written as $Y^*(t)$, which indeed is independent of $\xi$ by virtue of (iii), and \eqref{T***} follows directly. With the help of (i), we obviously have for all $(t,s)\in\triangle, h\ge0$, $x\in \R^d$ and $\omega\in\Omega,$
$$\phi(t+h, s, x,\omega)=\phi(t+h, t, \phi(t,s,x,\omega),\omega).$$
Whence, by approaching $s\to-\8$ and taking the continuous mapping theorem into account, the first identity in \eqref{T} follows directly.
Moreover, with the aid of \eqref{YYY}, we then similarly  obtain the second identity in \eqref{T}.

Based on the analysis above, to complete the proof of Proposition \ref{pr},
 it suffices to prove \eqref{Y}. To this end,
 we find that  for all $\xi=\xi_s\in L_\psi(\OO\to \R^d, \F_s,\P)$ and $r\ge0$,
\begin{equation}\label{Y*Y}
\begin{split}
\mathbb W_\psi\big(\mathscr L_{Y^{s-r,\xi}(t)},\mathscr L_{Y^{s,\xi}(t)}\big)&\le c_0\sum_{j=0}^{\lfloor r/\tau_0\rfloor}\mathbb W_\psi\big(\mathscr L_{Y^{s-(((j+1)\tau_0)\wedge r),\xi}(t)},\mathscr L_{Y^{s-j\tau_0,\xi}(t)}\big)\\
 &=c_0\sum_{j=0}^{\lfloor r/\tau_0\rfloor}\mathbb W_\psi\big(\mathscr L_{Y^{s-j\tau_0,Y^{s-(((j+1)\tau_0)\wedge r),\xi}(s-j\tau_0)}(t)},\mathscr L_{Y^{s-j\tau_0,\xi}(t)}\big)\\
 &\le c_0\sum_{j=0}^{\lfloor r/\tau_0\rfloor}h(t-s+j\tau_0)\mathbb W_\psi\big(\mathscr L_{Y^{s-(((j+1)\tau_0)\wedge r),\xi}(s-j\tau_0)},\mathscr L_{\xi}\big)\\
 &\le c_0\sum_{j=0}^\8h(t-s+j\tau_0)\sup_{r\in\R}\sup_{t\ge r }\mathbb W_\psi\big(\mathscr L_{Y^{r,\xi}(t)},\mathscr L_{\xi}\big),
\end{split}
\end{equation}
where in the first inequality we employed the weak triangle inequality due to \eqref{WU},   in the  identity we exploited the flow property, and the second inequality is owing to
 the contractive property in (iii). Thus, \eqref{Y} follows from \eqref{Y*Y} and by taking \eqref{L1*} into consideration.
\end{proof}

\begin{remark}
The condition (ii) in Proposition \ref{pr} is related to the uniform boundedness in the  sense of mean square for $\psi(Y^{s,\xi}(t),{\bf0})$, and moreover (iii) shows the contractive property of transition kernels of the stochastic process $Y^{s,\xi}(t)$ starting from   different initial distributions. So  far, there are plenty of probabilistic approaches (e.g. coupling methods) to examine (iii); see e.g. the monograph \cite{Chen}.  Obviously, the condition \eqref{L1*} is valid when the function $h$ is exponentially decay. Indeed, the condition \eqref{L1*}   allows   the function $h$ to be sub-exponentially decay in some setting.
\end{remark}

With the Proposition \ref{pr} at hand, we are in a position to complete the proof of Theorem \ref{main2}.
\begin{proof}[Proof of Theorem \ref{main2}]
To finish the proof of Theorem \ref{main2}, it is sufficient to examine the assumptions in Proposition \ref{pr} with suitable distance-like function $\psi$ (which is to be chosen later), one-by-one.
Under $({\bf A})$, in terms of \cite[Theorem 3.1.1.]{PR}, \eqref{H*} has a unique strong solution  $(X^{s,x}(t))_{t\ge s}$
 so that the Assumption (i) in Proposition \ref{pr}
is verifiable and therefore the stochastic semi-flow $\phi$ generated by the solution process  satisfies the property \eqref{YYY}; see e.g. \cite[p.34]{luo2014} for more details.

Below, we define  $P_{s,t}^*\mu=\mathscr L_{X^{s,\xi}(t)}$ if the $\F_s$-measurable random variable $\xi\sim\mu\in\mathscr P(\R^d)$. Once we can show that there exist constants $C,\ll>0$ such that   for all $(t,s)\in\triangle$ and $\mu,\nu\in\mathscr P_1(\R^{d})$,
\begin{equation}\label{HH*}
  \mathbb W_1\big(P_{s,t}^*\mu,P_{s,t}^*\nu\big)\le C\,\e^{-\ll \lfloor (t-s)/\tau\rfloor }\mathbb W_1\big(\mu,\nu\big),
\end{equation}
then the assumption (iii) in Proposition \ref{pr} follows directly with $\psi(x,y)=|x-y|$ and  $h(t)=C\e^\ll\e^{-\frac{\ll t}{\tau} }, t\ge0.$

Provided that  \eqref{HH*} is available,  we deduce that
\begin{equation*}\label{jj}
\begin{split}
\mathbb W_1\big(P_{s,t}^*\mu,\mu\big)
&\le  \sum_{i=0}^{\lfloor (t-s)/\tau\rfloor}\mathbb W_1\big(P_{s,(s+(i+1)\tau)\wedge t}^*\mu,P_{s,s+i\tau}^*\mu\big)\\
&=\sum_{i=0}^{\lfloor (t-s)/\tau\rfloor}\mathbb W_1\big(P_{s,s+i\tau}^*P_{s+i\tau,((s+(i+1)\tau))\wedge t}^*\mu,P_{s,s+i\tau}^*\mu\big)\\
&\le C\sum_{i=0}^{\lfloor (t-s)/\tau\rfloor}\e^{-\ll i\tau}\mathbb W_1\big(P_{s+i\tau,((s+(i+1)\tau))\wedge t}^*\mu,\mu\big)\\
&\le  \frac{C}{1-\e^{-\ll\tau}}\sup_{r\in[0,\tau]}\mathbb W_1\big(P_{s,  s+r}^*\mu,\mu\big),
 \end{split}
\end{equation*}
where in the first inequality
we used the triangle inequality,   in the identity we utilized the semigroup property of $P_{s,t}^*$, and  in the second inequality we applied \eqref{HH*}. Hence, to check the assumption (ii) in Proposition \ref{pr}, it is sufficient to show that
\begin{equation}\label{KK}
\sup_{r\in[0,\tau]}\mathbb W_1\big(P_{s,  s+r}^*\mu,\mu\big)<\8.
\end{equation}
By the triangle inequality, it follows that for $\xi\in L^1(\OO\to\R^d,\F_s,\P)$,
\begin{equation*}
   \sup_{r\in[0,\tau]}\mathbb W_1\big(P_{s,  s+r}^*\mu,\mu\big)\le  \sup_{r\in[0,\tau]}\E|X_{s+r}^{s,\xi}|+\int_{\R^d}|x|\,\mu(\d x).
\end{equation*}
Thus, \eqref{KK} follows by noting that for some $C_0>0,$
\begin{equation}\label{j4}
\sup_{r\in[0,\tau]}\E|X_{s+r}^{s,\xi}|\le C_0(1+\E|\xi|),
\end{equation}
which is a  more or less standard estimate under the Assumption ({\bf A}); see, for example, \cite{Mao}.

Set
$$\varphi(r):=\int_0^r\e^{-\int_0^u\gamma(v)\,\d v}\,\d u\int_u^\8l\e^{\int_s^l\,\gamma(v)\,\d v}\,\d l,\qquad r\ge 0,$$
with \begin{equation}\label{3j}
\gamma(v):=(K_1+K_2)v\,{\bf 1}_{\{0\le r\le L\}}-K_2v.\end{equation}
Apparently,  we have
\begin{equation}\label{6j}
\varphi'(r)=\e^{-\int_0^r\gamma(v)\,\d v}\int_r^\8l\e^{\int_0^l\,\gamma(v)\,\d v}\,\d l,\quad \varphi''(r)+\gamma(r)\varphi'(r)=-r,\quad r\ge 0.
\end{equation}
According to the expression of $\gamma(\cdot)$, it is easy to see that
\begin{equation*}
  0<C_*:=\inf_{r\ge 0}\varphi'(r)  \le \sup_{r\ge 0}\varphi'(r):=C^*<\8.
\end{equation*}
Then, by the mean value theorem, besides $\varphi(0)=0$, we find that
\begin{equation}\label{7j}
C_*r\le \varphi(r)\le C^*r,\qquad r\ge0.
\end{equation}
This definitely  implies
\begin{equation}\label{7jj}
C_*\mathbb W_1\le \mathbb W_\psi\le C^*\mathbb W_1,
\end{equation}
where $\psi(x,y):=\varphi(|x-y|), x,y\in\R^d.$
Whence, to derive \eqref{HH*}, it is sufficient to prove that
there exist constants $C_0,\ll>0$ such that   for all $(t,s)\in\triangle$ and $\mu,\nu\in\mathscr P_\psi(\R^{d})$,
\begin{equation}\label{eqn:HH*2}
  \mathbb W_\psi\big(P_{s,t}^*\mu,P_{s,t}^*\nu\big)\le C_0\,\e^{-\ll \lfloor (t-s)/\tau\rfloor }\mathbb W_\psi\big(\mu,\nu\big).
\end{equation}

Below, we adopt the reflection coupling strategy (see, for instance, \cite{LW} for time-homogeneous SDEs) to derive \eqref{eqn:HH*2} concerning with the time-periodic SDE \eqref{H*}.
Let
\begin{equation*}
\Pi(x)=I_{d\times d}-  \frac{2xx^*}{|x|^2} ,\qquad x\neq{\bf0},
\end{equation*}
where $I_{d\times d}$ means the $d\times d$ identity matrix, and $x^*$ denotes the transpose of $x.$ Obviously, for $x\neq{\bf0}$, $\Pi(x)$ is an orthogonal matrix.
Consider the following auxiliary SDE
\begin{equation}\label{H***}
 \d Y^{s,\eta}(t)=b(t,Y^{s,\eta}(t))\,\d t+\ss{\aa(t)}\,\Pi(X^{s,\xi}(t)-Y^{s,\eta}(t))\, \d W(t),\qquad s\le t<T_s,
\end{equation}
where $T_s$ is the coupling time defined by
$$T_s=\inf\{t\ge s: X^{s,\xi}(t)=Y^{s,\eta}(t)\}.$$
When $t\ge T_s$, we  stipulate  $X^{s,\xi}(t)=Y^{s,\eta}(t)$ based on the strong wellposedness of \eqref{H*}. For notation abbreviation, set $Z(t):=X^{s,\xi}(t)-Y^{s,\eta}(t)$.
From \eqref{H*} and \eqref{H***}, we have
\begin{equation*}\label{j}
\d Z(t)=\big(b(t,X^{s,\xi}(t))-b(t,Y^{s,\eta}(t))\big)\, \d t+2\ss{\aa(t)}\,\frac{Z(t)}{|Z(t)|} \d W^s(t),\qquad s\le t<T_s,
\end{equation*}
where
$$W^s(t):=\int_s^t\< Z(r)^*/|Z(r)|,\d W(r)\>$$
is a Brownian motion
via L\'evy's characterizations; see e.g.
\cite[Lemma 10.15, p.291]{DZ} for more details.
Thus, by noting that for $x\neq{\bf 0}$,
\begin{equation*}
\nn\varphi(|x|)=\varphi'(|x|)\frac{x}{|x|},\qquad \nn^2\varphi(|x|)=\varphi''(|x|)\frac{xx^*}{|x|^2}+\varphi'(|x|)\bigg(\frac{1}{|x|}I_{d\times d}-\frac{xx^*}{|x|^3}\bigg)
\end{equation*}
followed by
applying It\^o's formula, we derive that
\begin{equation}\label{j1}
\begin{split}
&\d \Big(\e^{\frac{1}{C^*}\int_s^t\aa(r)\,\d r}\varphi(|Z(t)|)\Big)\\
&=\e^{\frac{1}{C^*}\int_s^t\aa(r)\,\d r}\Big(\frac{\aa(t)}{C^*}\varphi(|Z(t)|)+\varphi'(|Z(t)|)\frac{1}{|Z(t)|}\<Z(t), b(t,X^{s,\xi}(t))-b(t,Y^{s,\eta}(t))\>\\
&\quad+2\aa(t)\varphi''(|Z(t)|)\Big)\,\d t+\d M^s(t)\\
&\le \aa(t)\e^{\frac{1}{C^*}\int_s^t\aa(r)\,\d r}\Big(\frac{1}{C^*}\varphi(|Z(t)|)-|Z(t)|\Big)\,\d t+\d M^s(t)\\
&\le \d M^s(t),\qquad s\le t<T_s,
\end{split}
\end{equation}
with $\d M^s(t):=\e^{\int_s^t\alpha(r)\,\d r}\ss{\varphi(t)}\,\d W^s(t)$,
where the first inequality is due to \eqref{HH} and \eqref{6j}
and the second inequality holds true thanks to \eqref{7j}. Thus, integrating from $s$ to $t\wedge T_s$ followed by taking expectations on both sides of \eqref{j1} yields
\begin{equation*}
  \E\Big(\e^{\frac{1}{C^*}\int_s^{t\wedge T_s}\aa(r)\,\d r}\varphi(|Z(t\wedge T_s)|)\Big)  \le \E\varphi(|\xi-\eta|).
\end{equation*}
This, along with the fact that
$$\E\Big(\e^{\frac{1}{C^*}\int_s^{t\wedge T_s}\aa(r)\,\d r}\varphi(|Z(t\wedge T_s)|)\Big)= \e^{\frac{1}{C^*}\int_s^{t}\aa(r)\,\d r}\E \varphi(|Z(t) |)  $$
in view of $\varphi(0)=0$, yields
$$\mathbb W_{\psi}\big(P_{s,t}^*\mu,P_{s,t}^*\nu\big)\le \E \varphi(|Z(t) |) \le \e^{-\frac{1}{C^*}\int_s^{t}\aa(r)\,\d r}\E\varphi(|\xi-\eta|).$$
Consequently, \eqref{HH*} follows from Lemma \ref{lem0} below and by choosing random variables $\xi,\eta\in L_{\psi}(\OO\to\R^d, \F_s,\P)$ such that $\mathbb W_\psi\big(\mu,\nu\big)=\E\varphi(|\xi-\eta|)$, which is true in terms of existence of optimal coupling.
\end{proof}

\section{Proof of Theorem \ref{thm:main1}}\label{sec:existence}
In this section, we aim to complete the proof of Theorem \ref{main2}. To achieve this, we first prepare  some warm-up work. More precisely, we intend to demonstrate that the functional solution  to \eqref{eq:SPDE} is uniformly bounded and is continuous w.r.t. the initial value
in the mean-square sense.

First of all, we prepare the following fundamental fact concerned with periodic functions.
\begin{lemma}\label{lem0}
Assume that $f:\R\to\R$ is a  $\tau$-periodic function. Then, for any $t\ge s$,
\begin{equation}\label{W1}
  \int_s^{t}f (u)\,\d u= \lfloor (t-s)/\tau\rfloor\int_0^\tau f(u)\,\d u+\int_{s-\lfloor s/\tau\rfloor \tau}^{t-(\lfloor s/\tau\rfloor+\lfloor (t-s)/\tau\rfloor) \tau}f(u)\,\d u,
\end{equation}
where, for a real number $a$, $\lfloor a\rfloor$ stands for its integer part.
\end{lemma}

\begin{proof}
By virtue of the $\tau$-periodic property of the function $f$, it is easy to see that for any $t\ge s$,
 \begin{equation*}
\begin{split}
  \int_s^t f (u)\,\d u  &=\int_s^{\lfloor (t-s)/\tau\rfloor\tau} f(u)\,\d u+\int_{\lfloor (t-s)/\tau\rfloor\tau}^{s+\lfloor (t-s)/\tau\rfloor\tau} f(u)\,\d u\\
  &\quad+\int_{s+\lfloor (t-s)/\tau\rfloor\tau}^t f \big(u-(\lfloor s/\tau\rfloor+\lfloor (t-s)/\tau\rfloor) \tau\big)\,\d u\\
  &=\int_s^{\lfloor (t-s)/\tau\rfloor\tau} f(u)\,\d u+\int_0^sf(u)\,\d u+\int_{s-\lfloor s/\tau\rfloor{\tau}}^{t-(\lfloor s/\tau\rfloor+\lfloor (t-s)/\tau\rfloor) \tau}f(u)\,\d u\\
  &=\int_0^{\lfloor (t-s)/\tau\rfloor\tau} f(u)\,\d u+\int_{s-\lfloor s/\tau\rfloor{\tau}}^{t-(\lfloor s/\tau\rfloor+\lfloor (t-s)/\tau\rfloor) \tau}f(u)\,\d u\\
  &=\lfloor (t-s)/\tau\rfloor\int_0^\tau f(u)\,\d u+\int_{s-\lfloor s/\tau\rfloor{\tau}}^{t-(\lfloor s/\tau\rfloor+\lfloor (t-s)/\tau\rfloor) \tau}f(u)\,\d u.
\end{split}
\end{equation*}
Therefore, \eqref{W1} follows right now.
\end{proof}

Next, we recall the following  Burkholder-Davis-Gundy (BDG for short) inequality for continuous martingales due to Osekowski \cite{Os}, where the sharp upper BDG's constant (i.e., $\chi$ below) is crucial in providing sufficient conditions to guarantee existence of random periodic solutions.

\begin{lemma}\label{ose}
  For any continuous martingale $M$ satisfying $M(0)=0$ and for any $t \ge 0$,
  $$
\E\bigg( \sup_{0 \le s \le t} M(s)\bigg)\le \chi\, \E\big( [M,M]_t^{1/2} \big),
$$
where $\chi\approx1.30693$  is the smallest positive root of the confluent hypergeometric function with parameter 1, and $[M,M]_t$ means the quadratic variation of $M(t)$.
\end{lemma}

\begin{lemma}\label{lem:boundedness1}
Assume Assumption  $({\bf H})$ and    condition \eqref{WE}.
Then, for any $s\in\R$ and $\xi=\xi_s\in L^2(\OO\to\C,\F_s,\P)$, there exists a constant $C>0$ independent of $s$ and  $\xi$ such that
\begin{equation} \label{E11}
   \sup_{s\in\R} \sup_{t\geq s} \E \|X_{t}^{s,\xi}\|_{\infty}^2\le C\big(1+\E\|\xi\|_\8^2\big).
\end{equation}
\end{lemma}
\begin{proof}
In terms of \eqref{WE}, there exists an $\vv\in(0,1)$ sufficiently small such that
\begin{equation} \label{E10}
   \int_0^\tau\big(\ll_1(r)+\vv  +2\e^{-c_*(r_0,\tau)+\vv r_0}(\ll_2(r)+\ll_3(r)+\vv+2\ll_3(r) \chi^2\e^{-c_*(r_0,\tau)+2(c^*(r_0,\tau)+\vv r_0)})\big)\,\d r<0.
\end{equation}
In the sequel, we shall fix $\vv>0$ satisfying \eqref{E10}.
According to \eqref{eqn:A1} and \eqref{eqn:A2}, for any $\xi\in\mathscr C$ it follows that
\begin{align}\label{E4}
\begin{split}
    2\langle b(t,\xi),\xi(0)\rangle
    &=2\langle b(t,\xi)-b(t,{\bf0}),\xi(0)\rangle+2\langle b(t,{\bf0}),\xi(0)\rangle\\
    &\leq \lambda_1(t) |\xi(0)|^2+\lambda_2(t) \|\xi\|^2_\infty+2|b(t,{\bf0})|\cdot|\xi(0)|\\
    &\leq (\lambda_1(t)+\vv) |\xi(0)|^2+\lambda_2(t) \|\xi\|^2_\infty+|b(t,{\bf0})|^2/\vv,
\end{split}
\end{align}
and that
\begin{align}\label{E5}
    \begin{split}
           \| \si(t,\xi)\|^2_{\rm HS}
    &=\| \si(t,\xi)-\si(t,{\bf0})\|^2_{\rm HS}+2\langle \si(t,\xi)-\si(t,{\bf0}),\si(t,{\bf0})\rangle_{\rm HS}+\| \si(t,{\bf0})\|^2_{\rm HS}\\
 &\le \Big(1+\frac{\vv}{\ll_3(t)}\Big) \| \si(t,\xi)-\si(t,{\bf0})\|^2_{\rm HS}+(1+\ll_3(t)/\vv)\| \si(t,{\bf0})\|_{\rm HS}^2\\
    &\leq (\lambda_3(t)+\vv) \|\xi\|^2_\8+(1+\ll_3(t)/\vv) \| \si(t,{\bf0})\|^2_{\rm HS}
    \end{split}
\end{align}
by exploiting the non-negative property of the function $\ll_3(\cdot)$.

Below, for notation brevity we
shall write $X^s_t$ and $X^s(t)$ instead of $X^{s,\xi}_t$ and $X^{s,\xi}(t)$, respectively.
Applying It\^o's formula  yields
\begin{align}\label{eqn:Ito1}
    \begin{split}
     \d\Big(\e^{-\int_s^t(\ll_1 (r)+\vv)\,\d r}|X^s(t)|^2\Big)&= \e^{-\int_s^t(\ll_1 (r)+\vv)\,\d r} \big(-(\ll_1 (t)+\vv) |X^s(t)|^2+2\langle X^s(t), b(t,X^s_t)\rangle\\ &\quad+\|\si(t,X^s_t\|^2_{\rm HS}\big)\,\d t+\d M(t),
\end{split}
\end{align}
in which
$$\d M^s(t):=2\e^{-\int_s^t(\ll_1 (r)+\vv)\,\d r}\langle X^s(t),\si(t,X^s_t) \mathrm{d}W(t)\rangle.$$
Then, combining \eqref{E4} with \eqref{E5},
we infer from \eqref{eqn:Ito1} that
\begin{equation}\label{E6}
\begin{split}
    \d\Big(\e^{-\int_s^t(\ll_1 (r)+\vv)\,\d r}|X^s(t)|^2\Big)&\leq \e^{-\int_s^t(\ll_1 (r)+\vv)\,\d r}((\lambda_2(t)+\ll_3(t)+\vv)\|X^s_t\|^2_\infty+C_\vv(t)) \, \d t\\
    &\quad+\d M^s(t),
    \end{split}
\end{equation}
where
\begin{equation}\label{A4}
C_\vv(t):=|b(t,{\bf0})|^2/\vv+(1+\ll_3(t)/\vv) \| \sigma(t,{\bf0})\|^2_{\rm HS}, \qquad t\ge s.
\end{equation}
Observe that
\begin{equation} \label{A1}
\begin{split}
   \e^{-\int_s^t(\ll_1 (r)+\vv)\,\d r}\|X_t^s\|^2_\infty
   &=\sup_{-r_0\le \theta\le 0}\Big(\e^{-\int_s^{t+\theta}(\ll_1 (r)+\vv)\,\d r}|X^s(t+\theta)|^2\e^{-\int_{t+\theta}^{t}(\ll_1 (r)+\vv)\,\d r}\Big)\\
    &\leq \e^{-c_*(r_0,\tau)}\sup_{t-r_0\le r \le t}\Big(\e^{-\int_s^r(\ll_1 (u)+\vv)\,\d u}|X^s(r)|^2 \Big)\\
    &=\e^{-c_*(r_0,\tau)}\sup_{ (t-r_0)\vee s\le r\le t}\Big(\e^{-\int_s^r(\ll_1 (u)+\vv)\,\d u}|X^s(r)|^2 \Big)\\
    &\qquad\vee\bigg( \e^{-c_*(r_0,\tau)}\sup_{ t-r_0\le r \le  (t-r_0)\vee s}\Big(\e^{-\int_s^r(\ll_1 (u)+\vv)\,\d u}|X^s(r)|^2 \Big)\bigg)\\
    &=:I_1(t)\vee I_2(t),
    \end{split}
\end{equation}
where in the inequality we used the fact  that
\begin{equation}\label{A2}
\begin{split}
   \int_{t+\theta}^{t}(\ll_1 (r)+\vv)\,\d r&=\int_{t-\lfloor t/\tau\rfloor \tau+\theta}^{t-\lfloor t/\tau\rfloor \tau}(\ll_1 (r)+\vv)\,\d r\\
   &\ge \inf_{0\le u\le \tau,-r_0\le \theta\le0}\int_{u+\theta}^u(\ll_1 (r)+\vv)\,\d r\ge c_*(r_0,\tau)
   \end{split}
\end{equation}
by making use of the $\tau$-periodic property of $\ll_1(\cdot)$ in the identity. In \eqref{A2},   $c_*(r_0,\tau)$ was defined in \eqref{WW}.
Again, via the periodic property of $\ll_1(\cdot)$, in addition to the locally integrable property of $\ll_1(\cdot)$, we find from Lemma \ref{lem0} that
 \begin{equation*}
\begin{split}
 \sup_{ t-r_0\le r \le  (t-r_0)\vee s}\bigg(-\int_s^r(\ll_1 (u)+\vv)\,\d u\bigg)
 &\le  C_0:= \sup_{t\in\R} \int_{t-r_0}^t|\ll_1 (r)+\vv|\,\d r<\8.
 \end{split}
\end{equation*}
Thus we have that
\begin{align}\label{eqn:I2}
    I_2(t)\leq \e^{-c_*(r_0,\tau)}\sup_{ t-r_0\le r \le  (t-r_0)\vee s} \e^{-\int_s^r(\ll_1 (u)+\vv)\,\d u}\sup_{ t-r_0\le r \le  (t-r_0)\vee s}|X^s(r)|^2 \leq \e^{C_0-c_*(r_0,\tau)}\|\xi\|^2_\infty.
\end{align}
As a result, substituting \eqref{eqn:I2} and the integration form of \eqref{E6} into \eqref{A1}
gives that
\begin{equation}\label{E8}
\begin{split}
     \e^{-\int_s^t(\ll_1 (r)+\vv)\,\d r}\E\|X_t^s\|^2_\infty
     &\leq\big( 1\vee\e^{C_0-c_*(r_0,\tau)}\big)\E\|\xi\|^2_\infty+\e^{-c_*(r_0,\tau)} \int_s^{t}C_\vv(r) \e^{-\int_s^r(\ll_1 (u)+\vv)\,\d u}  \,\d r\\
     &\quad+\e^{-c_*(r_0,\tau)}\int_s^{t}(\lambda_2(r)+\ll_3(r)+\vv) \e^{-\int_s^r(\ll_1 (u)+\vv)\,\d u} \E\|X_r^s\|^2_\infty \,\d r  \\
     &\quad+\e^{-c_*(r_0,\tau)}\E\bigg(\sup_{ (t-r_0)\vee s\le r\le t}M^s(r)\bigg).
\end{split}
\end{equation}

The following crucial step is to estimate the fourth term on the right hand side of \eqref{E8}. To achieve this, we apply BDG's inequality (see Lemma \ref{ose} above for more details) to derive that
\begin{equation*}
\begin{split}
\E\bigg(\sup_{ (t-r_0)\vee s\le r\le t}M^s(r)\bigg)&=2\mathbb{E}\bigg(\sup_{(t-r_0)\vee s\le r\le t}\int_s^r \e^{-\int_s^u(\ll_1 (v)+\vv)\,\d v}\langle X^s(u),\si(u,X^s_u) \mathrm{d}W(u)\rangle\bigg)\\
&=2\mathbb{E}\bigg(\sup_{(t-r_0)\vee s\le r\le t}\int_{(t-r_0)\vee s}^r \e^{-\int_s^u(\ll_1 (v)+\vv)\,\d v}\langle X^s(u),\si(u,X^s_u) \mathrm{d}W(u)\rangle\bigg)\\
&\leq 2\chi  \mathbb{E}\bigg(\int_{(t-r_0)\vee s}^{t} \e^{-2\int_s^r(\ll_1 (u)+\vv)\,\d u} |\si(r,X^s_r)^*X^s(r)|^2\, \d r\bigg)^{1/2},
\end{split}
\end{equation*}
where, for a matrix $A$, $A^*$ means its transpose.
Subsequently, via  Young's inequality, we find   that
\begin{align}
&\E\bigg(\sup_{ (t-r_0)\vee s\le r\le t}M^s(r)\bigg)\nonumber\\
&\le2\chi \mathbb{E} \bigg(\sup_{(t-r_0)\vee s\le r\le t}\Big(\e^{-\big(\int_s^t(\ll_1 (u)+\vv)\,\d u+\int_t^r(\ll_1 (u)+\vv)\,\d u\big)}|X^s(r)|^2\Big)\nonumber\\
&\quad\times\int_{(t-r_0)\vee s}^{t} \e^{-\int_s^r(\ll_1 (u)+\vv)\,\d u}  \|\si(r,X^s_r)\|^2_{\rm HS}\, \d r\bigg)^{1/2} \nonumber\\
&\le 2\chi\e^{\sup_{(t-r_0)\vee s\le r\le t}\int_r^t(\ll_1 (u)+\vv)\,\d u}\nonumber\\
&\quad\times\mathbb{E} \bigg(\e^{-\int_s^t(\ll_1 (u)+\vv)\,\d u}\|X^s_t\|^2_\8 \int_{(t-r_0)\vee s}^{t} \e^{-\int_s^r(\ll_1 (u)+\vv)\,\d u}  \|\si(r,X^s_r)\|^2_{\rm HS}\, \d r\bigg)^{1/2} \label{A3}\\
&\le2\chi\e^{c^*(r_0,\tau)+\vv r_0}  \mathbb{E} \bigg(\e^{-\int_s^t(\ll_1 (u)+\vv)\,\d u}\|X^s_t\|^2_\8 \int_{(t-r_0)\vee s}^{t} \e^{-\int_s^r(\ll_1 (u)+\vv)\,\d u}  \|\si(r,X^s_r)\|^2_{\rm HS}\, \d r\bigg)^{1/2} \nonumber \\
&\le \frac{1}{2}\e^{c_*(r_0,\tau)}\e^{-\int_s^t(\ll_1 (u)+\vv)\,\d u}\E\|X^s_t\|^2_\8+2\chi^2\e^{-c_*(r_0,\tau) +2(c^*(r_0,\tau)+\vv r_0)}\nonumber\\
&\quad\times \int_s^t\e^{-\int_s^r(\ll_1 (u)+\vv)\,\d u}\E\|\si(r,X^s_r)\|^2_{\rm HS}\, \d r\nonumber\\
&\le\frac{1}{2}\e^{c_*(r_0,\tau)}\e^{-\int_s^t(\ll_1 (u)+\vv)\,\d u}\E\|X^s_t\|^2_\8\nonumber\\
&\quad+2\chi^2\e^{-c_*(r_0,\tau)+2(c^*(r,\tau)+\vv r_0)}\int_s^t(\ll_3(r)+\vv)\e^{-\int_s^r(\ll_1 (u)+\vv)\,\d u}\E \|X^s_r\|^2_\8\, \d r\nonumber\\
&\quad+2\chi^2\e^{-c_*(r_0,\tau)+2(c^*(r_0,\tau)+\vv r_0)}\int_s^t(1+\ll_3(r)/\vv)\| \si(r,{\bf0})\|^2_{\rm HS}\e^{-\int_s^r(\ll_1 (u)+\vv)\,\d u}  \, \d r,\nonumber
\end{align}
where in the third inequality we utilized
\begin{equation*}
\begin{split}\sup_{(t-r_0)\vee s\le r\le t}\int_r^t(\ll_1 (u)+\vv)\,\d u&\le \sup_{ -r_0 \le \theta\le 0}\int_{t+\theta}^t(\ll_1 (u)+\vv)\,\d u
=\sup_{ -r_0 \le \theta\le 0} \int_{t-\lfloor t/\tau\rfloor \tau+\theta}^{t-\lfloor t/\tau\rfloor \tau}(\ll_1 (r)+\vv)\,\d r\\
&\le \sup_{0\le u\le \tau,-r_0\le \theta\le0}\int_{u+\theta}^u(\ll_1 (r)+\vv)\,\d r=c^*(r_0,\tau)+\vv r_0
\end{split}
\end{equation*}
with $c^*(r_0,\tau)$ being  introduced in \eqref{WW}, and in the last display we used \eqref{E5}.  Next, substituting the estimate \eqref{A3} into \eqref{E8} yields
\begin{equation*}
\begin{split}
     \e^{-\int_s^t(\ll_1 (r)+\vv)\,\d r}\E\|X_t^s\|^2_\infty
     &\leq \e^{C_0-c_*(r_0,\tau)}\E\|\xi\|^2_\infty+\e^{-c_*(r_0,\tau)} \int_s^{t}C_\vv(r) \e^{-\int_s^r(\ll_1 (u)+\vv)\,\d u}  \,\d r\\
     &\quad+\e^{-c_*(r_0,\tau)}\int_s^{t}\lambda_2(r) \e^{-\int_s^r(\ll_1 (u)+\vv)\,\d u} \E\|X_r^s\|^2_\infty \,\d r  \\
     &\quad+ \frac{1}{2}\e^{-\int_s^t(\ll_1 (u)+\vv)\,\d u}\E\|X^s_t\|^2_\8\\
&\quad+2\chi^2\e^{-2(c_*(r_0,\tau)+2(c^*(r_0,\tau)+2\vv r_0)}\int_s^t(\ll_3(r)+\vv)\e^{-\int_s^r(\ll_1 (u)+\vv)\,\d u}\E \|X^s_r\|^2_\8\, \d r\\
&\quad+2\chi^2\e^{-2c_*(r_0,\tau)+2(c^*(r_0,\tau)+\vv r_0)}\int_s^t(1+\ll_3(r)/\vv)\| \si(r,{\bf0})\|^2_{\rm HS}\e^{-\int_s^r(\ll_1 (u)+\vv)\,\mathrm{d} u}  \, \d r.
\end{split}
\end{equation*}
Consequently, we arrive at
\begin{equation*}
\begin{split}
    \e^{-\int_s^t(\ll_1 (r)+\vv)\,\d r}\E\|X_t^s\|^2_\infty
     &\leq 2\e^{C_0-c_*(r_0,\tau)}\E\|\xi\|^2_\infty+ \int_s^{t}\LL_2^\vv(r) \e^{-\int_s^r(\ll_1 (u)+\vv),\d u}  \,\d r\\
     &\quad+\int_s^{t}\LL_1^\vv(r) \e^{-\int_s^r(\ll_1 (u)+\vv)\,\d u} \E\|X_r^s\|^2_\infty \,\d r,
\end{split}
\end{equation*}
where for $C_\vv(\cdot)$   given in \eqref{A4},
\begin{equation*}
\begin{split}
    \LL_1^\vv(t):&=2\e^{-c_*(r_0,\tau)}\big(\lambda_2(t)+\ll_3(t)+\vv+2\chi^2\e^{-c_*(r_0,\tau)+2(c^*(r_0,\tau)+\vv r_0)}(\ll_3(t)+\vv)\big),\\
    \LL_2^\vv(t):&=2\e^{-c_*(r_0,\tau)}C_\vv(t)+4\chi^2\e^{-2c_*(r_0,\tau)+2(c^*(r_0,\tau)+\vv r_0)}(1+\ll_3(t)/\vv)\| \si(r,{\bf0})\|^2_{\rm HS}.
\end{split}
\end{equation*}
Then, making use of Gronwall's inequality (e.g., \cite[Theorem 1.20, p.18]{Kle}) gives \begin{equation*}
\begin{split}
\E\|X_t^s\|^2_\infty
     &\leq 2\e^{C_0-c_*(r_0,\tau)}\E\|\xi\|^2_\infty \e^{\int_s^t(\ll_1 (r)+\vv)\,\d r}+ \int_s^{t}\LL_2^\vv(r) \e^{\int_r^t(\ll_1 (u)+\vv)\,\d u}  \,\d r\\
     &\quad+2\e^{C_0-c_*(r_0,\tau)}\E\|\xi\|^2_\infty\e^{\int_s^t(\ll_1 (r)+\vv)\,\d r}\int_s^{t}\LL_1^\vv (r)\e^{\int_r^t\LL_1^\vv (u)\,\d u}\,\d r\\
     &\quad+\int_s^{t}\bigg( \int_s^r\LL_2^\vv(u) \e^{\int_u^t(\ll_1 (v)+\vv)\,\d v}  \,\d u\bigg)\LL_1^\vv (r)\e^{\int_r^t\LL_1^\vv (u)\,\d u}\,\d r\\
     &=:\Pi_1(t)+\Pi_2(t)+\Pi_3(t)+\Pi_4(t).
\end{split}
\end{equation*}
In what follows, we estimate $\Pi_i(t), i=1,2,3,4,$ one by one. Since the function $\ll_1(\cdot)$ is $\tau$-periodic, we therefore obtain from Lemma \ref{lem0} that
\begin{equation*}
\begin{split}
  \int_s^t(\ll_1 (r)+\vv)\,\d r
  &\le \lfloor (t-s)/\tau\rfloor\int_0^\tau(\ll_1 (r)+\vv)\,\d r+2\int_0^\tau|(\ll_1 (r)+\vv)|\,\d r.
  \end{split}
\end{equation*}
Hence, $\Pi_1(\cdot)$ is bounded as below:
\begin{equation}\label{WW1}
\Pi_1(t)\le  2\e^{C_0-c_*(r_0,\tau)}\E\|\xi\|^2_\infty \exp\bigg(\lfloor (t-s)/\tau\rfloor\int_0^\tau(\ll_1 (r)+\vv)\,\d r+2\int_0^\tau|(\ll_1 (r)+\vv)|\,\d r\bigg).
\end{equation}
%\yuebox{Did we introduce the norm $\|\cdot\|_{\tau,\8}$ before?}
Next, by invoking Lemma \ref{lem0} again, it follows that
\begin{equation}\label{WW2}
\Pi_2(t)\le\|\LL_2^\vv\|_{\tau,\8}\e^{2(\|\ll_1 \|_{\tau,\8}+\vv)\tau}\int_s^t
\exp\bigg(\lfloor(t-r)/\tau\rfloor\int_0^\tau(\ll_1 (u)+\vv)\,\d u\bigg)\,\d r.
\end{equation}
Once more, using the integration-by-parts formula and taking advantage of Lemma \ref{lem0} yields
\begin{equation}\label{WW3}
\begin{split}
\Pi_3(t)&= 2\e^{C_0-c_*(r_0,\tau)}\E\|\xi\|^2_\infty\e^{\int_s^t(\ll_1 (r)+\vv)\,\d r}\Big( \e^{\int_s^t\LL_1^\vv (r)\,\d r}-1\Big)\\
&\le 2\e^{C_0-c_*(r_0,\tau)}\E\|\xi\|^2_\infty\\
&\quad\times\exp\bigg(\lfloor (t-s)/\tau\rfloor\int_0^\tau((\ll_1 (r)+\vv)+\LL_1^\vv (r))\,\d s+2\tau(\|\ll_1+\LL_1^\vv \|_{\tau,\8}+\vv)\bigg).
 \end{split}
\end{equation}
Finally,
via Fubini's theorem and the integration-by-parts formula, we obtain that
\begin{equation*}
\begin{split}
\Pi_4(t)&=\int_s^{t} \LL_2^\vv(u) \e^{\int_u^t(\ll_1(v)+\vv) \,\d v} \bigg(\int_u^t \LL_1^\vv (r)\e^{\int_r^t\LL_1^\vv (u)\,\d u}\,\d r\bigg)\,\d u\\
&=\int_s^{t} \LL_2^\vv(r) \e^{\int_r^t(\ll_1(u)+\vv)\,\d u} \bigg(\e^{\int_r^t\LL_1^\vv (u)\,\d u}-1\bigg) \,\d r.
\end{split}
\end{equation*}
Thus, Lemma \ref{lem0} enables us to derive that
\begin{equation*}
\begin{split}
\Pi_4(t)\le \|\LL_2^\vv\|_{\tau,\8}\e^{2\tau(\|\ll_1 +\LL_1^\vv \|_{\tau,\8}+\vv)}\int_s^t
\exp\bigg(\lfloor(t-r)/\tau\rfloor\int_0^\tau(\ll_1(u)+\vv+\LL_1^\vv (u))\,\d u\bigg)\,\d r.
\end{split}
\end{equation*}
Combining this with \eqref{WW1}, \eqref{WW2}, and \eqref{WW3},  we find that for some constant $C>0$ (independent of $s,t$ and $\xi$),
\begin{equation*}
\begin{split}
\E\|X_t^s\|^2_\infty
     &\leq
     C\int_s^t
\exp\bigg(\lfloor(t-r)/\tau\rfloor\int_0^\tau(\ll_1(u)+\vv)\,\d u\bigg)\,\d r\\
     &\quad+C\E\|\xi\|^2_\infty\exp\bigg( \lfloor (t-s)/\tau\rfloor \int_0^\tau(\ll_1(r)+\vv+\LL_1^\vv (r))\,\d r \bigg)\\
     &\quad+C\int_s^t
\exp\bigg(\lfloor(t-r)/\tau\rfloor\int_0^\tau(\ll_1(u)+\vv+\LL_1^\vv (u))\,\d u\bigg)\,\d r.
\end{split}
\end{equation*}
This further implies that for some constant $C_0>0$ (independent of $s,t$ and $\xi$),
\begin{equation*}
\begin{split}
\E \|X_t^s\|^2_\infty
     &\leq
     C_0\int_s^t
\exp\bigg((t-r)/\tau\int_0^\tau(\ll_1(u)+\vv)\,\d u\bigg)\,\d r\\
     &\quad+C_0\E \|\xi\|^2_\infty \exp\bigg(   (t-s)/\tau  \int_0^\tau(\ll_1(r)+\vv+\LL_1^\vv (r))\,\d r \bigg)\\
     &\quad+C_0\int_s^t
\exp\bigg( (t-r)/\tau \int_0^\tau(\ll_1(u)+\vv+\LL_1^\vv (u))\,\d u\bigg)\,\d r.
\end{split}
\end{equation*}
As a consequence, \eqref{E11} follows directly from \eqref{E10}.
\end{proof}

In the sequel,  we aim to explore the continuous dependence in the mean-square sense concerning initial values.
\begin{lemma}\label{lem:stable1}
Under Assumption $({\bf H})$, for all  $(t,s)\in\triangle$ and $\xi=\xi_s,\eta=\eta_s\in L^2(\OO\to\C,\F_s,\P)$, there exists a constant $C^*>0$ such that
\begin{equation}\label{E2}\E\|X_t^{s,\xi}-X_t^{s,\eta}\|^2_\8\le C^* \E\|\xi-\eta\|^2_\infty\exp\bigg(\lfloor (t-s)/\tau\rfloor\int_0^\tau\Theta (r)\,\d r \bigg),
\end{equation}
where for all $r\in\R,$
$$\Theta(r):=\ll_1(r)+2\e^{-c_*(r_0,\tau)}\big(\lambda_2(r) +\ll_3(r)+2\chi^2\e^{-c_*(r_0,\tau)+2c^*(r_0,\tau)}\ll_3(r)\big).$$
Consequently, if   \eqref{WE} holds true,
then there exists a constant $C^{**}>0$ such that for all $t\ge s$,
\begin{equation}\label{E3}
 \E\|X_t^{s,\xi}-X_t^{s,\eta}\|^2_\8\le C^{**}\e^{-\ell (t-s)}\E\|\xi-\eta\|^2_\infty,
\end{equation}
where $\ell>0$ was given in \eqref{WE}.
\end{lemma}
\begin{proof}
For notation brevity, we set
\begin{equation*}\label{FF}
  \Lambda_t^s:=X_t^{s,\xi}-X_t^{s,\eta},\qquad t\ge s
\end{equation*}
so $\Lambda^s(t):=X^{s,\xi}(t)-X^{s,\eta}(t)$.
Applying It\^o's formula   and making use of \eqref{eqn:A1} and \eqref{eqn:A2} yields
\begin{align}\label{eqn:Ito2}
    \begin{split}
     \d\big(\e^{-\int_s^t\ll_1(r)\,\d r}|\LL^s(t)|^2\big)&= \e^{-\int_s^t\ll_1(r)\,\d r} \big(-\ll_1(t)|\LL^s(t)|^2+2\< \LL^s(t), b(t,X^{s,\xi}_t)-b(t,X^{s,\eta}_t)\> \\
     &\quad \quad +\|\si(t,X^{s,\xi}_t)-\si(t,X^{s,\eta}_t)\|^2_{\rm HS}\big)\,\d t+\d N^s(t)\\
     &\leq (\ll_2(t)+\ll_3(t))\e^{-\int_s^t\ll_1(r)\,\d r} \|\LL^s_t\|_\8^2\,\d t+\d N^s(t),
\end{split}
\end{align}
where
$$\d N^s(t):=2\e^{-\int_s^t\ll_1(r)\,\d r}\< \LL^s(t),(\si(t,X^{s,\xi}_t)-\si(t,X^{s,\eta}_t))\,\d W(t)\>. $$
By following the argument to derive \eqref{E8}, we have
\begin{equation}\label{3P}
\begin{split}
     \e^{-\int_s^t\ll_1 (r)\,\d r}\E\|\LL^s_t\|^2_\infty
     &\leq \e^{C_0-c_*(r_0,\tau)}\E\|\xi-\eta\|^2_\infty\\
     &\quad+\e^{-c_*(r_0,\tau)}\int_s^{t}(\lambda_2(r)+\ll_3(r)) \e^{-\int_s^r\ll_1 (u)\,\d u} \E\|\LL_r^s\|^2_\infty \,\d r  \\
     &\quad+\e^{-c_*(r_0,\tau)}\E\bigg(\sup_{ (t-r_0)\vee s\le r\le t}N^s(r)\bigg).
\end{split}
\end{equation}
Next, by carrying out a similar way to derive \eqref{A3}, we infer that
\begin{equation*}
\begin{split}
\E\bigg(\sup_{ (t-r_0)\vee s\le r\le t}N^s(r)\bigg)
&\le\frac{1}{2}\e^{c_*(r_0,\tau)}\e^{-\int_s^t\ll_1 (r)\,\d r}\E\|\LL^s_t\|^2_\8\\
&\quad+2\chi^2\e^{-c_*(r_0,\tau)+2c^*(r_0,\tau)}\int_s^t\ll_3(r)\e^{-\int_s^r\aa(u)\,\d u}\E \|\LL^s_r\|^2_\8\, \d r.
\end{split}
\end{equation*}
Now, plugging this estimate back into \eqref{3P} yields
\begin{equation}
\begin{split}
     \e^{-\int_s^t\ll_1 (r)\,\d r}\E\|\LL^s_t\|^2_\infty
     &\leq 2\e^{C_0-c_*(r_0,\tau)}\E\|\xi-\eta\|^2_\infty+\int_s^{t}\hat{\lambda}(r) \e^{-\int_s^r\ll_1 (u)\,\d u} \E\|\LL_r^s\|^2_\infty \,\d r ,
\end{split}
\end{equation}
where
\begin{equation*}
    \hat{\ll}(t):=2\e^{-c_*(r_0,\tau)}\big(\lambda_2(t)+\ll_3(t) +2\chi^2\e^{-c_*(r_0,\tau)+2c^*(r_0,\tau)}\ll_3(t)\big).
\end{equation*}
Subsequently, applying the Gronwall inequality and taking Lemma \ref{lem0} into consideration enables the following estimate
\begin{equation*}
\begin{split}\E\|\LL^s_t\|^2_\infty&\le 2\e^{C_0-c_*(r_0,\tau)}\E\|\xi-\eta\|^2_\infty \e^{\int_s^t(\ll_1(r)+\hat{\ll} (r))\,\d r}\\
&\le C_1\E\|\xi-\eta\|^2_\infty\exp\bigg(\lfloor (t-s)/\tau\rfloor \int_0^\tau(\ll_1(s)+\hat{\ll} (s))\,\d s\bigg)\\
&\le C_2\E\|\xi-\eta\|^2_\infty\exp\bigg(  (t-s)/\tau  \int_0^\tau(\ll_1(s)+\hat{\ll} (s))\,\d s \bigg)\\
\end{split}
\end{equation*}
for some constants $C_1,C_2>0$,
so \eqref{E2} holds true.
Finally, with   \eqref{WE} at hand, the desired assertion \eqref{E3} is available immediately.
 \end{proof}

As aforementioned, $(\OO,\F,\P,\theta)$ is a metric dynamical system and the path map $\phi$ defined in \eqref{WWW} is a stochastic semi-flow. If we further have for all $(s,t)\in\triangle$, $h\in\R$, $\xi\in\C$ and $\omega\in\OO,$
\begin{equation}\label{WWP}
 \phi(t+h,s+h,\xi,\omega)=\phi(t,s,\xi,\theta_\tau\omega),
\end{equation}
then the stochastic semi-flow $\phi$ corresponds to a random dynamical system. Due to the influence of time inhomogeneity,  the relationship \eqref{WWP} need not hold for all $h\in\R$. As Lemma \ref{lemma3} below
 shows that  the relation \eqref{WWP} holds true merely for $h=\tau$, which however is sufficient for our goal in the present work.

\begin{lemma}\label{lemma3}
Under Assumption $({\bf H})$, for all $t\ge s$, $\xi\in\C$ and $\omega\in\Omega$,
\begin{equation}\label{PP4}
 \phi(t+\tau,s+\tau,\xi,\omega)=\phi(t,s,\xi,\theta_\tau\omega).
\end{equation}
\end{lemma}

\begin{proof}
To verify \eqref{PP4}, it is sufficient to check that the functional solution $(X_t^{s,\xi})_{t\ge s}$ to \eqref{eq:SPDE} satisfies that
 for  all $t\ge s$, $\xi\in\C$ and $\omega\in\OO$,
$$X_{t+\tau}^{s+\tau ,\xi}(\omega)=X_t^{s,\xi}(\theta_\tau\omega).$$
Obviously,  we  deduce from  \eqref{eq:SPDE} that  for all $t\ge s$, $\xi\in\C$ and $\omega\in\OO$,
\begin{equation}\label{PP1}
    X^{s,\xi}(t,\omega)=\xi(0)+\int_{s}^{t}b(u,X_u^{s,\xi}(\omega))\,\d u+\int_{s}^{t}\si(u,X^{s,\xi}(u,\omega))\,\d\omega(u).
\end{equation}
Again, from the functional SDE \eqref{eq:SPDE}, we find that all $t\ge s$, $\xi\in\C$ and $\omega\in\OO$,
\begin{equation}\label{PP}
\begin{split}
    &X^{s,\xi}(t,\theta_\tau\omega)\\&=\xi(0)+\int_s^tb(u,X_u^{s,\xi}(\theta_\tau\omega))\,\d u+\int_{s}^{t}\si(u,X^{s,\xi}(u,\theta_\tau\omega))\,\d(\theta_\tau\omega(u))\\
   &=\xi(0)+\int_s^tb(u+\tau,X_u^{s,\xi}(\theta_\tau\omega))\,\d u+\int_{s}^{t}\si(u+\tau,X^{s,\xi}(u,\theta_\tau\omega))\,\d(\theta_\tau\omega(u)) \\
   &=\xi(0)+\int_{s+\tau}^{t+\tau}b(u,X_{u-\tau}^{s,\xi}(\theta_\tau\omega))\,\d u+\int_{s+\tau}^{t+\tau}\si(u,X^{s,\xi}(u-\tau,\theta_\tau\omega))\,\d(\theta_{-\tau}\theta_\tau\omega(u))\\
   &=\xi(0)+\int_{s+\tau}^{t+\tau}b(u,X_{u-\tau}^{s,\xi}(\theta_\tau\omega))\,\d u+\int_{s+\tau}^{t+\tau}\si(u,X^{s,\xi}(u-\tau,\theta_\tau\omega))\,\d \omega(u),
    \end{split}
\end{equation}
where in the second identity we used the $\tau$-periodic property of $b$ and $\si$ w.r.t. the time variable, in the third identity we utilized the strategy of variable substitution, and in the last identity  we took advantage of the group property of the shift operator and $\theta_0={\rm id}_\OO.$ Now, for  all $u\ge s$, $\xi\in\C$ and $\omega\in\OO$,  let's define $Y_\cdot^{s+\tau,\xi}(\omega)\in\C$ by
\begin{equation}\label{PP3}
    Y_u^{s+\tau,\xi}(\omega)=X_{u-\tau}^{s,\xi}(\theta_\tau\omega),
\end{equation}
which apparently implies  $Y^{s+\tau,\xi}(u,\omega)=X^{s,\xi}(u-\tau,\theta_\tau\omega)$. With the notation $Y_\cdot^{s+\tau}(\omega)$ at hand, \eqref{PP} can be reformulated as
\begin{equation}\label{PP2}
\begin{split}
    Y^{s+\tau,\xi}(t+\tau,\omega)
   &=\xi(0)+\int_{s+\tau}^{t+\tau}b(u,Y^{s+\tau,\xi}_u(\omega))\,\d u+\int_{s+\tau}^{t+\tau}\si(u,Y^{s+\tau,\xi}(u,\omega))\,\d \omega(u).
    \end{split}
\end{equation}
Once more, by virtue of strong wellposedness of functional solutions to \eqref{eq:SPDE}, we deduce from \eqref{PP1} and \eqref{PP2} that
for  all $t\ge s$, $\xi\in\C$ and $\omega\in\OO$
$$X_{t+\tau}^{s+\tau,\xi}(\omega)=Y^{s+\tau,\xi}_{t+\tau}(\omega).$$
This, together with \eqref{PP3}, yields \eqref{PP4} directly.
\end{proof}

Prior to proceeding to the proof of Theorem \ref{thm:main1}, we present a refined framework of \cite[Theorem 3.2.4]{luo2014} to allow weaker asymptotic conditions, which underpins the existence and uniqueness of pathwise random periodic solutions of stochastic semi-flows.

\begin{proposition}\label{pr2} For   a Markov process $(Y^{s,x}(t))_{t\ge s}$ on the Polish space $(\mathbb U, \rho)$ and some $p>0$,
assume that
\begin{enumerate}
\item[$(i)$]$(Y^{s,\xi}(t))_{t\ge s}$
enjoys the semi-flow property;

\item[$(ii)$]For all $\xi=\xi_s\in L^p(\OO\to\mathbb U, \F_s,\P)$, there exists a constant $ C_0(\xi)>0$ such that
\begin{equation*}
 \sup_{t\ge s} \E\rho(Y^{s,\xi}(t), {\bf0}\big)^p\le C_0(\xi).
\end{equation*}

\item[$(iii)$] There exists a   function $h:[0,\8)\to[0,\8)$ satisfying for each fixed $t\in\R$ and some $\tau_0>0$,
\begin{equation}\label{L1}
    \lim_{s\to-\8}\sum_{j=0}^\8h(t-s+j\tau_0)=0
\end{equation}
and such that
 for all $(t,s)\in\triangle$ and $ \xi=\xi_s,\eta=\eta_s\in  L^p(\OO\to\mathbb U, \F_s,\P)$,
\begin{equation}\label{TT4T}
\Big(\E\rho\big( Y^{s,\xi}(t),Y^{s,\eta}(t)\big)^p\Big)^{\frac{1}{1\vee p}}\le h(t-s)\Big(\E\rho( \xi,\eta)^p\Big)^{\frac{1}{1\vee p}}.
\end{equation}
\end{enumerate}
Then, if the map $\phi:\triangle \times \mathbb U\times \OO\to \mathbb U$ defined via
\begin{equation*}
\phi(t,s,x,\omega)=Y^{s,x}(t,\omega), \quad (t,s)\in\triangle, ~x\in \mathbb U, ~\omega\in\Omega,
\end{equation*}
further satisfies the following property: for all $(t,s)\in\triangle, x\in \mathbb U,$   $\omega\in\Omega,$ and some $\tau>0,$
\begin{equation}\label{YYY*}
\phi(t+\tau,s+\tau,x,\omega)=\phi(t,s,x,\theta_\tau\omega),
\end{equation}
there exists a unique   $\F $-measurable stochastic process $(Y^*(t))_{t\in\R}$ such that for all $t\in\R$, and   $r\ge0,$
\begin{equation}\label{T*}
Y^*(t+r,\omega)=\phi(t+r, t, Y^*(t,\omega),\omega),\qquad Y^*(t+\tau,\omega)=Y^*(t,\theta_\tau\omega) \quad {\rm a.s. }
\end{equation}
and moreover  for all $\xi=\xi_s\in L^p(\OO\to\mathbb U,\F_s,\P)$ with $s\in\R,$
 \begin{equation}\label{P-}
\lim_{s\to-\8}  \E\rho(Y^{s,\xi}(t),Y^*(t))^p=0.
 \end{equation}
\end{proposition}
\begin{proof}
Define the metric between  $\xi\in L^p(\OO\to\mathbb U)$ and $\eta\in L^p(\OO\to\mathbb U)$ as below
\begin{equation*}
d_p(\xi,\eta)=\big(\E\rho(\xi,\eta)^p\big)^{\frac{1}{1\vee p}},\qquad p>0.
\end{equation*}
Due to (ii), it is easy to see that $d_p(Y^{s,\xi}(t),{\bf0})$ is well defined
for all $(t,s)\in\triangle$ and $\xi=\xi_s\in L^p(\OO\to\mathbb U,\F_s,\P)$.  Once $Y^{\cdot,\xi}(t)$ is a Cauchy sequence under the metric $d_p$, i.e.,   for all $(t,s)\in \triangle$ and $\xi=\xi_s\in L^p(\OO\to\mathbb U,\F_s,\P)$,
\begin{equation}\label{LL}
\lim_{s\to-\8}\sup_{r\ge0}d_p\big(Y^{s-r,\xi}(t),Y^{s,\xi}(t)\big)=0.
\end{equation}
then there exists a unique stochastic process $(Y^*(t))_{t\in\R}\in L^p(\OO\to\mathbb U,\F,\P)$, which is independent of $\xi$, such that
\begin{equation}\label{P*}
 \lim_{s\to-\8} d_p(Y^{s,\xi}(t),Y^*(t))  =0
\end{equation}
so \eqref{P-} follows. Next,   we find that
for all $(t,s)\in\triangle$, $\xi=\xi_s\in L^p(\OO\to\mathbb U,\F_s,\P)$ and $r\ge0$, \begin{equation}\label{P--}
\begin{split}
d_p\big(\phi(t+r, t, Y^*(t)),Y^*(t+r)\big)&\le d_p\big(\phi(t+r, s,\xi),Y^*(t+r)\big)\\
&\quad+d_p\big(\phi(t+r, t, Y^*(t)),\phi(t+r, s,\xi )\big)\\
&= d_p\big(\phi(t+r, s,\xi ),Y^*(t+r)\big)\\
&\quad+ d_p\big(\phi(t+r, t, Y^*(t)),\phi(t+r,t,\phi(t,s,\xi))\big)\\
&\le d_p\big(\phi(t+r, s,\xi ),Y^*(t+r)\big)+h(r)d_p\big(  Y^*(t), \phi(t,s,\xi) \big)\\
&\longrightarrow0\qquad \mbox{ as }\quad s\to-\8,
\end{split}
\end{equation}
where in the first inequality we used the triangle inequality, in the identity we utilized the flow property, in the second inequality we made use of \eqref{TT4T}, and the last display  is due to \eqref{P*}.
With \eqref{YYY*} at hand, by following the argument to derive \eqref{P--}, we also have
\begin{equation*}
 d_p\big(Y^*(t+\tau,\cdot),Y^*(t,\theta_\tau\cdot)\big)=0
\end{equation*}
This, in addition to \eqref{P--}, yields \eqref{T*} via .
Therefore, to complete the proof of Proposition \ref{pr2},
 it is sufficient to verify \eqref{LL}.
Indeed,
by the triangle inequality and the flow property of $Y^{s,\xi}(t)$, note that \begin{equation*}
 \begin{split}
 d_p\big(Y^{s-r,\xi}(t),Y^{s,\xi}(t)\big)&\le \sum_{j=0}^{\lfloor r/\tau_0\rfloor}d_p\big(Y^{s-(((j+1)\tau_0)\wedge r),\xi}(t),Y^{s-j\tau_0,\xi}(t)\big)\\
 &=\sum_{j=0}^{\lfloor r/\tau_0\rfloor}d_p\big(Y^{s-j\tau,Y^{s-(((j+1)\tau_0)\wedge r),\xi}(s-j\tau_0)}(t),Y^{s-j\tau_0,\xi}(t)\big)\\
 &\le \sum_{j=0}^{\lfloor r/\tau_0\rfloor}h(t-s+j\tau)d_p\big(Y^{s-(((j+1)\tau_0)\wedge r),\xi}(s-j\tau_0),\xi\big)\\
 &\le \sum_{j=0}^\8h(t-s+j\tau_0)\sup_{r\in\R}\sup_{t\ge r }d_p\big(Y^{r,\xi}(t),\xi\big),
  \end{split}
\end{equation*}
where the first inequality is owing to the triangle inequality, the identity holds true thanks to the flow property and  the last two inequality is available by \eqref{TT4T}.
Consequently, \eqref{LL} follows by  taking (ii) and \eqref{L1} into consideration.
\end{proof}

With the help of Lemmas \ref{lem:boundedness1}-\ref{lemma3} and Proposition \ref{pr2}, we can  complete  the proof of Theorem \ref{thm:main1}.
\begin{proof}[Proof of Theorem \ref{thm:main1}]
In terms of Proposition \ref{pr2}, to complete the proof of Theorem \ref{thm:main1},
it is sufficient to examine that for  all $t\ge s$ and $\xi=\xi_s,\eta=\eta_s\in L^2(\OO\to\C,\F_s,\P)$,
\begin{itemize}
\item[(a)]  $\phi(t+\tau,s+\tau,\xi,\omega)=\phi(t,s,\xi,\theta_\tau\omega)$ for all $\omega\in\OO$;

\item[(b)] There exists a decreasing function $h:[0,\8)\to(0,\8)$ satisfying  \eqref{L1} and such that
$$\E\|\phi(t,s,\xi)-\phi(t,s,\eta)\|_\8^2\le h(t-s)\E\|\xi-\eta\|_\8^2;$$

\item[(c)] $\phi$ is ultimately  bounded in the mean-square sense, i.e., there is a constant $C>0$ (independent of $\xi$) such that
$$\sup_{t\ge s}\E\|\phi(t,s,\xi)\|_\8^2\le C(1+\E\|\xi\|_\8^2).$$
\end{itemize}
Trivially, (a) holds true owing to Lemma \ref{lemma3}. On the other hand, (b) is valid by taking advantage of Lemma \ref{lem:stable1}, which is obtained via the synchronous coupling (i.e., the same functional SDEs driven by the same Brownian motion but with different initial values),  and choosing $h(t)=C^{**}\e^{-\ell t}$ for all $t\in\R$. Furthermore, (iii) is verifiable by invoking Lemma \ref{lem:boundedness1}.
\end{proof}

The following  section is devoted to the proof of Theorem \ref{thm3}. First of all, we prepare the following auxiliary lemma.

\begin{lemma}\label{ll3}
Under the assumptions of Theorem \ref{thm3}, there exist constants $C_1,C_2,\ll>0$ such that for all $(t,s)\in\triangle$ and $\xi=\xi_s,\eta=\eta_s\in L^2(\OO\to\C_{\aa_0},\F_s,\P)$,
\begin{equation}\label{B-}
   \E\|X_t^{s,\xi}\|_{\aa_0}^2 \le C_1\big(1+\E\|\xi\|_{\aa_0}^2\big),
\end{equation}
and
\begin{equation}\label{B--}
\E\|X_t^{s,\xi}-X_t^{s,\eta}\|_{\aa_0}^2\le C_2\e^{-\ll (t-s)}\E\|\xi-\eta\|_{\aa_0}^2.
\end{equation}
\end{lemma}

\begin{proof}
Herein, we  prove merely \eqref{B--} since \eqref{B-} can be handled similarly by combining  the argument  of Lemma \ref{lem:boundedness1}  with that of \eqref{B--}.
According to the definition of $\|\cdot\|_{\aa_0}$, it is easy to see that  for any $t\ge s,$
\begin{equation*}
\begin{split}
\e^{-\int_s^t\ll_1(u)\,\d u}  \|\LL^s_t\|_{\aa_0}^2
&=\bigg(\e^{-\int_s^t\ll_1(u)\,\d u}\sup_{s-t\le \theta\le0}\big(\e^{2\aa_0\theta}|\LL^s(t+\theta)|^2\big)\bigg)\\
&\qquad\vee\bigg(\e^{-\int_s^t\ll_1(u)\,\d u}\sup_{-\8<\theta\le s-t}\big(\e^{2\aa_0\theta}|\LL^s(t+\theta)|^2\big)\bigg)\\
&=:I_1(t)\vee I_2(t),
\end{split}
\end{equation*}
where $\LL^s_t:=X_t^{s,\xi}-X_t^{s,\eta}$. From \eqref{BB},
 we deduce that
\begin{equation}\label{2BB}
\begin{split}
I_1(t)&=\sup_{s-t\le \theta\le0}\Big(\e^{-\int_{t+\theta}^t(\ll_1(u)+2\aa_0)\d u}\e^{-\int_s^{t+\theta}\ll_1(u)\,\d u}|\LL^s(t+\theta)|^2\Big)\\
&\le \e^{\ll_{\aa_0,\tau}}\sup_{s\le u\le t}\Big( \e^{-\int_s^u\ll_1(v)\,\d v}|\LL^s(u)|^2\Big),
\end{split}
\end{equation}
where in the inequality we used the following fact that
\begin{equation}\label{3B}
\begin{split}
-\int_{t+\theta}^t(\ll_1(u)+2r)\,\d u
&=\frac{1}{\tau}\big(\theta-\theta-\lfloor -\theta/\tau\rfloor\tau\big)\int_0^\tau(\ll_1(u)+2\aa_0)\,\d u\\
&\quad-\int_{t+\theta-\lfloor(t+\theta)/\tau}^{t+\theta-\lfloor (t+\theta)/\tau\rfloor\tau-\theta-\lfloor-\theta/\tau\rfloor\tau}(\ll_1(u)+2\aa_0)\,\d u
\end{split}
\end{equation}
by taking  Lemma \ref{lem0} into consideration.
 Next, we derive that
\begin{equation}\label{4B}
\begin{split}
I_2(t)&=\e^{-\int_s^t(\ll_1(u)+2\aa_0)\,\d u}\sup_{-\8<\theta+t-s\le 0}\big(\e^{2\aa_0(\theta+t-s)}|\LL^s(s+\theta+t-s)|^2\big)\\
&=\e^{-\int_s^t(\ll_1(u)+2\aa_0)\,\d u}\sup_{-\8<\theta\le 0}\big(\e^{2\aa_0 \theta }|\LL^s(s+\theta)|^2\big)\\
&=\e^{-\int_s^t(\ll_1(u)+2\aa_0)\,\d u} \|\xi-\eta\|_{\aa_0}^2\\
&\le c_0\|\xi-\eta\|_{\aa_0}^2
\end{split}
\end{equation}
 for some constant $c_0>1,$ where in the third identity we exploited the notion of $\|\cdot\|_{\aa_0}$  and in the inequality we employed Lemma \ref{lem0} and
\eqref{BB1}.
Subsequently, combining \eqref{2BB} with \eqref{4B} yields
\begin{equation}\label{7B}
\e^{-\int_s^t\ll_1(u)\,\d u}  \|\LL^s_t\|_{\aa_0}^2\le  \bigg(\e^{\ll_{\aa_0,\tau}}\sup_{s\le u\le t}\Big( \e^{-\int_s^u\ll_1(v)\,\d v}|\LL^s(u)|^2\Big)\bigg)\vee\big(c_0 \|\xi-\eta\|_{\aa_0}^2\big).
\end{equation}
By It\^o's formula and BDG's inequality (i.e., Lemma \ref{ose}), it follows from \eqref{AA1} and \eqref{AA2} that
\begin{equation*}
\begin{split}
\E\bigg(\sup_{s\le u \le t}\Big(\e^{-\int_s^u\ll_1(v)\,\d v}  |\LL^s(u)|^2\Big)\bigg)
&\le \E\|\xi-\eta\|_{\aa_0}^2+\int_s^t(\ll_2(u)+\ll_3(u))\E\Big(\e^{-\int_s^u\ll_1(v)\,\d v} \|\LL^s_u\|^2_{\aa_0}\Big)\,\d u\\
&\quad+ \frac{1}{2}\E\bigg(\sup_{s\le u \le t}\Big(\e^{-\int_s^u\ll_1(v)\,\d v}  |\LL^s(u)|^2\Big)\bigg)\\
&\quad+2\chi^2\int_s^t\ll_3(u)\e^{-\int_s^u\ll_1(v)\,\d v}  \|\LL^s_u\|_{\aa_0}^2\,\d  u.
\end{split}
\end{equation*}
That is,
\begin{equation*}
\begin{split}
\E\bigg(\sup_{s\le u \le t}\Big(\e^{-\int_s^u\ll_1(v)\,\d v}  |\LL^s(u)|^2\Big)\bigg)
&\le 2\E\|\xi-\eta\|_{\aa_0}^2\\
&\quad+2\int_s^t\big(\ll_2(u)+(1+2\chi^2)\ll_3(u)\big)\e^{-\int_s^u\ll_1(v)\,\d v}  \|\LL^s_u\|_{\aa_0}^2\,\d  u.
\end{split}
\end{equation*}
This, together with \eqref{7B}, implies
\begin{equation*}
\begin{split}
\e^{-\int_s^t\ll_1(u)\,\d u}
\E\|\LL^s_t\|_{\aa_0}^2&\le \big(2\,\e^{\ll_{\aa_0,\tau}}\vee c_0\big)\E\|\xi-\eta\|_{\aa_0}^2\\
&\quad+2\e^{\ll_{\aa_0,\tau}}\int_s^t\big(\ll_2(u)+(1+2\chi^2)\ll_3(u)\big)\E\Big(\e^{-\int_s^u\ll_1(v)\,\d v} \|\LL^s_u\|^2_{\aa_0}\Big)\,\d u.
\end{split}
\end{equation*}
Consequently, \eqref{B--} follows from   Gronwall's inequality and by taking  \eqref{B*}
into account.
\end{proof}

With the aid of Lemma \ref{ll3}, we complete the
\begin{proof}[Proof of Theorem \ref{thm3}]
By implementing a similar argument of Lemma \ref{lemma3}, we infer that the stochastic semi-flow, defined in \eqref{J*}, satisfies \eqref{YYY*} with $\mathbb U=\C_{\aa_0}$. Moreover, according to Lemma \ref{ll3} above, we find that (ii) and (iii) in Proposition \ref{pr2} with $p=2$ and $\rho(\xi,\eta)=\|\xi-\eta\|_{\aa_0}, \xi,\eta\in\C_{\aa_0}$,
hold, respectively, with $h(t)=C\e^{-\ll t}$  for some constants $C,\ll>0$, which obviously satisfies \eqref{L1}. Therefore, the proof of Theorem \ref{thm3} is finished by applying  Proposition \ref{pr2} with  $\mathbb U=\C_{\aa_0}$ and $p=2$.
\end{proof}

\section{Appendix}
In this section, to make the content of this work self-contained,
let's
 recall the definition of  random periodic solution for stochastic semi-flows given in \cite{rpssde2011}.
Let $H$ be a separable Banach space and  $(\Omega,\F,\mathbb{P},(\theta_s)_{s\in \mathbb{R}})$ a metric dynamical system.  Consider a stochastic semi-flow $u: \Delta \times \Omega\times H\to H$, which satisfies the following standard condition:
  for  all    $ r\leq s\leq t,\ r, s,t\in \mathbb{R}$ and  $ \omega\in\Omega$,
\begin{eqnarray*}{\label{16}}
u(t,r,\omega)=u(t,s,\omega)\circ u(s,r,\omega) .
\end{eqnarray*}
We remark that
  the map $u(t,s,\omega): H\to H$ need not to be invertible for $(t,s)\in \Delta$ and $\omega \in \Omega$.

\begin{definition}
A  pathwise  random  periodic solution with period $\tau$ of the semi-flow  $u: \Delta\times H\times \Omega\to H$
 is an $\F$-measurable
map $y:\mathbb{R}\times \Omega\to H$ such that for   a.e. $\omega\in \Omega$,
\begin{eqnarray*}
u(t,s, y(s,\omega), \omega)=y(t,\omega),\qquad
y(s+\tau,\omega)=y(s, \theta_\tau \omega),\qquad  (t,s)\in  \triangle.
\end{eqnarray*}
\end{definition}

\begin{definition}
A    random  $\tau$-periodic solution in the sense of distribution of the semi-flow  $u: \Delta\times H\times \Omega\to H$
 is an $\F$-measurable
map $y:\mathbb{R}\times \Omega\to H$ such that for all $ (t,s)\in  \triangle$,
\begin{eqnarray*}
u(t,s, y(s,\omega), \omega)\overset{d}{=}y(t,\omega),\qquad
y(s+\tau,\omega)\overset{d}{=}y(s, \theta_\tau \omega).
\end{eqnarray*}
\end{definition}

\end{document}